\documentclass[12pt]{amsart}
\usepackage{graphicx,comment}
\usepackage{color,comment}
\usepackage{amsthm}
\usepackage{amsmath}
\usepackage{amsfonts}
\usepackage{enumerate}
\usepackage{amssymb}
\usepackage{epstopdf}
\usepackage{hyperref}
\usepackage{xypic,tikz,graphicx}

\usepackage{pgfplots}
\pgfplotsset{compat=1.11}
\usepgfplotslibrary{fillbetween}
\usetikzlibrary{intersections,patterns}
\usetikzlibrary{decorations.markings}
\usetikzlibrary{decorations.pathmorphing}

\pgfdeclarelayer{bg}
\pgfdeclarelayer{fg}
\pgfsetlayers{bg,main,fg}

\newtheorem{theorem}{Theorem}[section]
\newtheorem{proposition}[theorem]{Proposition}
\newtheorem{proposition*}{Proposition}
\newtheorem{lemma}[theorem]{Lemma}

\newtheorem{conjecture}[theorem]{Conjecture}

\newtheorem{mainthm}{Theorem}

\DeclareMathOperator{\G}{\Gamma}
\DeclareMathOperator{\Om}{\Omega}
\DeclareMathOperator{\Si}{\Sigma}
\DeclareMathOperator{\Area}{Area}
\DeclareMathOperator{\diam}{diam}
\DeclareMathOperator{\length}{\text{\rm L}}

\DeclareMathOperator{\F}{\mathbf{F}}

\newcommand{\R}{\mathbb{R}}

\newcommand{\eps}{\varepsilon}

\newcommand{\finalbound}{22500}

\usetikzlibrary{decorations.markings}

\newlength\mylen

\tikzset{
tricolor/.style n args={3}{
  decoration={
    markings,
    mark=at position 0.5 with {
      \node[draw=none,inner sep=0pt,fill=none,text width=0pt,minimum size=0pt] {\global\setlength\mylen{\pgfdecoratedpathlength}};
    },
  },
  draw=#1,
  dash pattern=on 0.333333\mylen off 0.666666\mylen,
  preaction={decorate},
  postaction={
    draw=#2,
    dash pattern=on 0.333333\mylen off 0.333333\mylen,dash phase=0.333333\mylen
  },
  postaction={
    draw=#3,
    dash pattern=on 0.333333\mylen off 0.666666\mylen,dash phase=0.333333\mylen
    },
  }
}

\theoremstyle{definition}
\newtheorem{definition}[theorem]{Definition}

\newtheorem{remark}[theorem]{Remark}

\title[Length of geodesic in PSC manifold]{Length of a closed geodesic in 3-manifolds of positive scalar curvature}

\author{Yevgeny Liokumovich}
\address{Department of Mathematics, University of Toronto}
\author{Davi Maximo}
\address{Department of Mathematics, University of Pennsylvania}
 \author{Regina Rotman}
\address{Department of Mathematics, University of Toronto}

\begin{document}
	
	\maketitle
	
	\begin{abstract} 
 %Let $M$ be a closed $3$-dimensional Riemannian manifold with a positive scalar curvature, $R_g \geq \Lambda_0 >0$. In this paper we will compute a constant $C$, for which  the length of a shortest non-trivial closed geodesic on $M$ can always be bounded by $\frac{C}{\sqrt{\Lambda_0}}$, thus proving a conjecture of M. Gromov in dimension $3$.  
 Let $M$ be a closed $3$-dimensional Riemannian manifold with positive scalar curvature, $R_g \geq 6$. We show that $M$ contains a non-trivial
 closed geodesic of length less than $\finalbound$. This confirms a conjecture of M. Gromov in dimension $3$.
	\end{abstract}

\section{Introduction}
Closed geodesics are basic objects of study in Riemannian Geometry, appearing  in other fields of mathematics, such as Topology and Dynamical Systems, as well as Physics.

Let $M^n$ be a closed Riemannian manifold. The existence of a periodic geodesic was established in 1951 by A. Fet and L. Lusternik by using Morse theory on the space $\Lambda M^n$ of closed piecewise differentiable curves on $M^n$ \cite{LyuFet}. It is natural to want to estimate the length of a shortest closed geodesic in terms of other geometric parameters of $M^n$, such as its volume, diameter and curvature. Generally speaking, it is easier to bound the length of a shortest closed geodesic on non-simply connected closed Riemannian manifolds.  For example, it is easy to see that the length of the shortest closed geodesic on non-simply connected closed Riemannian manifolds is always bounded above by twice the diameter of the manifold. In 1949 K. Loewner found a sharp area upper bound for the length of the shortest closed geodesic on a Riemannian $2$-torus, followed by 1951 result of C. Pu establishing a sharp area upper bound for the length of the shortest closed geodesic on a Riemannian real projective plane. These results gave rise to Systolic Geometry. See ~\cite{CroK2003},
~\cite{Katz2007} for a survey of the field.  

In his foundational paper ~\cite{gromov1983}, M. Gromov asked if there exists a constant $c(n)$, such that the length of a shortest closed geodesic, $l(M^n)$ can always be bounded from above by $c(n) \cdot vol(M^n) ^{\frac{1}{n}}$, where $vol(M^n)$ denotes the volume of $M^n$. In the same paper he proved his famous systolic inequality for essential manifolds: a class of manifolds, which, in particular, include closed Riemannian manifolds that admit a metric on non-positive sectional curvature. Recently, the constant in this inequality was improved to $c(n) =n$ by A. Nabutovsky in \cite{nab2022}, building on results from
\cite{papasoglu}, \cite{LLNR}. One can likewise  ask whether $l(M^n)$ can always be uniformly bounded in terms of the diameter of $M^n$. Note that it was demonstrated by F. Balacheff, C. Croke and M. Katz in 
~\cite{balcrokekatz2009}, unlike the non-simply connected case, when $M^n$ is simply connected it is not always possible to bound $l(M^n)$ by twice the diameter. 

When $M^n$ is a simply connected closed Riemannian manifold, it is difficult to establish curvature-free upper bounds for $l(M^n)$. In fact, the only case when such upper bounds are known to exist for a simply connected manifold is that of a Riemannian $2$-sphere. The first such upper bounds were established by C. Croke in ~\cite{croke1988}. They were subsequently improved by A. Nabutovsky and the third author in \cite{nabrot2002} and independently by S. Sabourau in ~\cite{sab2004} to $4d$ and $8\sqrt{A}$, where $d$ is the diameter and $A$ is the area of the manifold. The area bound was further improved to $4\sqrt{2A}$ by the third author in ~\cite{rot2006}. Also, in ~\cite{adelpal2022} I. Adelstein and F. Vargas Pallete proved that on any non-negatively  curved Riemannian $2$-sphere, there exists a closed geodesic of length at most $3d$. 

A classical theorem of Toponogov \cite{Toponogov}
states that simple closed geodesics
 in a Riemannian two-sphere with sectional curvature $K_g \geq 1$ have length at most 
$2 \pi$. In dimension $n >2$, curvature bounds are helpful in finding upper bounds for $l(M^n)$ (see ~\cite{nabrot2003}, ~\cite{wuzhu2022}), and to have a positive curvature (sectional, $e.g$ \cite{BTZ82}) seems to be particularly useful. For a closed Riemannian manifold $M^n$ with a positive Ricci curvature, $Ric \geq \frac{n-1}{r^2}$,  it was proven by the third author that 
$l(M^n) \leq 8\pi n r$ in  ~\cite{rot2024}. The bound was improved to $\pi n r$ by H.-B. Rademacher in ~\cite{rad2024}.

The main result of our paper provides a bound of this nature in dimension 3, replacing the lower Ricci curvature bound with a much weaker assumption of lower scalar curvature bound. 

\begin{mainthm}\label{thm:mainB}
Let $(M^3,g)$ be a closed manifold with scalar curvature $R_g\geq 6 $.
Then $M$ contains a non-trivial closed geodesic of length at most
$\finalbound$.
\end{mainthm}
This confirms a conjecture of Gromov
\cite[Conjecture (a)]{Gr2018} in dimension $3$. 
The bound in Theorem \ref{thm:mainB} is not sharp but it reveals an underlying rigidity in the space of 3-manifolds with $R_g \geq 6$, a condition which is flexible enough to allow for complicated geometric behavior such as splines (gravitational wells) and drawstrings, see $e.g.$~\cite{kazaras2023}.

\begin{figure}[h!]
\centering	
\includegraphics[width=75mm,scale=0.4]{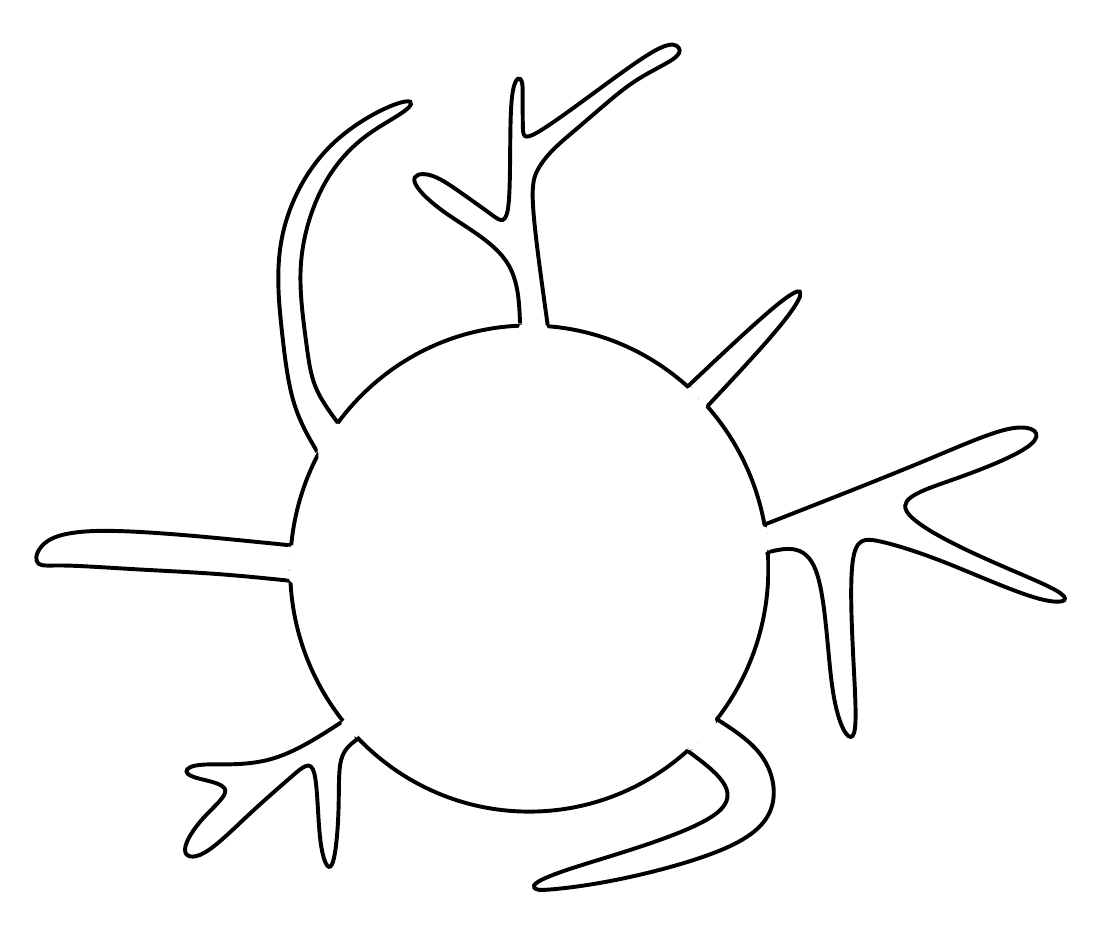}
\caption{A 3-sphere with many splines and $R\geq 6$ }
\end{figure}

Our proof will rely on the existence theorem of Fet and Lusternik. Their agument can be summarized as follows.

Let $\Omega M^n$ be the space of piecewise differentiable curves on $M^n$. From homotopy theory equations we know that there exists an integer $q$, such that $\pi_q\Omega (M^n) \neq \{0\}$, while $\pi_q(M^n) = \{0\}$. Let us consider a non-contractible map $f:S^q \longrightarrow \Omega M^n$. We will try to deform $f$ to a map whose image lies in the space of constant curves $ \Omega_0 M^n$ along the integral curves of the negative gradient of the energy functional on $\Omega M^n$. We should not be able to do it as it would contradict the non-contractibility of $f$. Thus, we obtain a closed geodesic as an obstruction to this extension. 

From this proof we can see that if we can estimate the length of curves in a non-contractible sphere $S^q \subset \Omega M^n$ in terms of some geometric parameters of $M^n$ then we will immediately obtain an upper bound for the length of a shortest closed geodesic. 

Note that it is not always possible to obtain such sweepout. For example, in ~\cite{lio2013}, the first author constructed a sequence of metrics on $S^2$ for which optimal sweepouts cannot be controlled solely in terms of the diameter of the sphere. The construction is based on a result of S. Frankel and M. Katz, in which they have defined a family of metrics on a $2$-disk with fixed diameter, but such that they require an uncontrollable length increase to contract the boundary of the disk to a point (see ~\cite{frakatz1993}). Likewise, O. Alshawa and H. Y. Cheng have constructed a sequence of metrics on $S^3$ of fixed diameter and volume where the optimal sweepouts must have arbitrarily long curves (see ~\cite{alscheng2024}).

Hence, one requires curvature assumptions to guarantee existence of a sweepout by short closed curves. In \cite{liozhou}
it was shown that 3-manifolds with positive Ricci curvature and bounded volume admit a sweepout by short 1-cycles. Note, however, that
the minimax method applied to the space of 1-cycles may not produce a geodesic, but only a stationary geodesic net.

It may also be possible to obtain a short closed geodesic on an $n$-dimensional sphere even if the sweepout by short curves does not exist using the following dichotomy: either one obtains a sweepout or shows that the existence of a sweepout by short curves is obstructed precisely by the existence of short closed geodesics of Morse index less than $n-1$.

However, in this paper we show that a sweepout of a Riemannian $3$-sphere by short closed curves exists and, thus, the closed geodesic we obtain will be a so called min-max geodesic with Morse index $2$ for generic Riemannian metrics. 

\begin{mainthm}\label{thm:mainC}
Let $(S^3,g)$ be a Riemannian 3-sphere with scalar curvature $R_g\geq \Lambda_0>0$.
Then there exists a non-contractible map $F: (I^2, \partial I^2) \longrightarrow (\Omega S^3, \Omega_0 S^3)$
with the length of $F(x)$ bounded by $\frac{\finalbound \sqrt{6}}{\sqrt{\Lambda_0}}$ for all $x$.
Hence, $(S^3,g)$ contains a non-trivial closed geodesic $\gamma$ of length at most
$\frac{\finalbound \sqrt{6}}{\sqrt{\Lambda_0}}$ and Morse index $\leq 2$. Moreover, if the metric $g$
is bumpy then $\gamma$ has Morse index $2$.
\end{mainthm}

\subsection{Proof overview} 
In Section \ref{non-sphere} we use topological classification
of compact 3-manifolds with positive scalar curvature to reduce Theorem \ref{thm:mainB}
to the case of 3-spheres. In the rest of the paper we construct a non-contractible
2-parameter family of short closed curves, proving Theorem \ref{thm:mainC}, which
via Morse theory on the free loop space \cite{bott1982} implies Theorem \ref{thm:mainB}
for $M \cong S^3$.

Our construction of a family of short closed curves on the $3$-sphere
can be roughly described as the following dimension reduction argument (with one important caveat
we explain below): first we construct
a family of ``small" 2-spheres foliating $M$, then we construct a sweepout of each 2-sphere 
in the foliation by short curves with the property that the sweepout changes continuously
as we vary the 2-spheres. 

Now we explain the caveat to the strategy described above. It has to do with the notion
of ``smallness" for 2-spheres that we need to construct a sweepout by short
closed curves. It is shown in \cite{LNR} that we can construct a sweepout of 
a 2-sphere by curves controlled in terms of the area and diameter of the 2-sphere
(but neither area only nor diameter only bounds are sufficient). Hence, the strategy
in the previous paragraph can be realized if we could prove the existence of a foliation
of $M$ by 2-spheres of small area and \emph{intrinsic} diameter. We state this result as a conjecture below
in terms of tree-foliations (see
 Definition \ref{def: tree-foliation}).

\begin{conjecture} \label{conjecture}
    There exists a constant $C>0$ with the following property.
    Let $(S^3,g)$ be a Riemannian 3-sphere with scalar curvature $R_g\geq \Lambda_0>0$.
Then there exists a tree-foliation $\{ \Sigma_t \}_{t \in T}$ of $(S^3,g)$
with 
\begin{enumerate}
    \item $\diam(\Sigma_t) \leq \frac{C}{\sqrt{\Lambda_0}}$;
    \item $\Area(\Sigma_t) \leq \frac{C}{\Lambda_0}$.
\end{enumerate}
Here $\diam(\Sigma_t)=\sup_{x,y}\{dist_{\Sigma_t}(x,y):x,y\in \Sigma_t\}$ denotes the diameter of $\Sigma_t$ in the intrinsic metric on $\Sigma_t$.
\end{conjecture}

In \cite{LiMa}, Mean Curvature Flow with Surgery was used to show that a 3-sphere with positive scalar curvature
admits a tree-foliation by 2-spheres of controlled area and \emph{ambient}
diameter. Unfortunately, this is not enough to guarantee existence of a sweepout by short curves.
To any 2-sphere $\Sigma$ embedded in $M$ one can attach three very long ``fingers" (or ``spikes") that wrap around $\Sigma$ inside $M$, so that they stay in a small neighbourhood
of $\Sigma$. This operation increases the diameter of $\Sigma$ in the ambient metric by an arbitrarily small amount,
but any sweepout of such a sphere by closed curves will contain a curve that climbs up and down one of the spikes and hence will have very large length
(examples like this of Riemannian metrics on the 2-sphere resembling a three-legged starfish were considered
in \cite[Remark 4.10]{sab2004}, \cite[Fig. 1]{liokumovich2014surfaces}).

While we are not able to prove Conjecture \ref{conjecture},
or prove a bound for any other geometric invariant of 2-spheres in 
a foliation that would imply existence of a sweepout by short curves of each sphere, 
we found a way to control a different geometric invariant that we call {\it $L$-shortness}. 
$L$-shortness guarantees that for every closed curve $\gamma$ on a surface
$\Sigma$ embedded in $M$ there exists a continuous family of arcs of length at most $L$ in $M$ connecting
the points of $\gamma$ to a fixed point $\gamma(0)$. Note that the curves may not 
necessary lie on $\Sigma$. We then use deformation techniques that were developed in \cite{NR13} to prove that $L$-shortness of an embedded $2$-sphere $\Sigma$ implies that an arbitrary sweepout of $\Sigma$ can be deformed to a family of closed curves of length controlled in terms of $L$ and, moreover, this deformation can be made continuously with respect to deformations of $\Sigma$. 

Hence, our proof can be summarized as follows. First, in Section \ref{sec: L-short} we construct a tree-foliation of $M$
by $L$-short $2$-spheres for $L$ bounded above in terms of the lower bound for scalar curvature of $M$. Then in Section \ref{sec: parametric}
we consider an arbitrary family of sweepouts of these $2$-spheres (together forming a two-parameter family of closed curves in $M$) and show that $L$-shortness can be used to deform this two-parameter family of (possibly very long) curves to a family of curves of controlled length. In order to prove existence of a tree-foliation by $L$-short $2$-spheres in Section \ref{sec: L-short} we develop a certain combinatorial analog of the Mean Curvature Flow that works by deforming a given mean convex surface along certain minimal $2$-discs in a way that decreases the area by a definite amount after each deformation.

\subsection{Acknowledgments} Y. L. was supported by NSERC Discovery grant. R.R. was supported by NSERC Discovery grant. R. R. gratefully acknowledges the support from the Princeton IAS Summer Collaborators program in 2024 and  the Simons Laufer Mathematical Sciences Institute (formerly MSRI) in Berkeley, California, during the Fall 2024 semester, where part of this paper was completed. Part of this paper was completed while Y. L. and R. R. were at the Metric Geometry trimester program at the Hausdorff Research Institute for Mathematics; they are grateful to the Institute for its hospitality. D.M. was partially supported by NSF grant DMS-1910496.  Part of this paper was written while R. R. was on a sabbatical leave. During this time she was partially supported by the Simons Foundation International Award SFI-MPS-SFM-00006548.

\section{$3$-manifolds non-diffeomorphic to the $3$-sphere} \label{non-sphere}

We first observe that is enough to prove Theorem \ref{thm:mainB} for orientable manifolds.
If $(M^3,g)$ is non-orientable and $\gamma$ is a short geodesic in the orientable double cover of $M$, then
the image of $\gamma$ under the covering map is a short geodesic in $(M,g)$. 
Suppose $(M^3,g)$ is a closed orientable 3-manifold with  scalar curvature $R\geq 6$. By the topological classification of 3-manifolds with positive scalar curvature  (\cite{SY79}, \cite{GrLa83}, \cite{perelman2002entropy},
\cite{perelman2003finite},
\cite{perelman2003ricci})
$M$ must be diffeomorphic to
a connect sum of 
$$(S^3/\Gamma_1)\#\cdots\#(S^3/\Gamma_k)\# (S^2\times S^1)\#\cdots \#(S^2\times S^1)$$
where $\Gamma_1,\ldots,\Gamma_k$ are finite subgroups of $SO(4)$. Thus, assuming $M$ is non-diffeomorphic to the 3-sphere, it implies that $M$ is either reducible or a spherical space form $S^3/\Gamma$, for some nontrivial $\Gamma$. We consider the former case first: 

\begin{comment}
For a minimal $2$-sphere in a 3-manifold $M$ with positive scalar curvature
we have the following area and diameter estimates:
\begin{theorem}[Area and diameter of minimal 2-sphere] \label{minimal sphere}
    Suppose $(M^3,g)$ is a closed 3-manifold with  scalar curvature $R_g\geq \Lambda_0>0$ and $\Sigma \subset M^3$ is a minimal embedded 2-sphere in $(M,g)$.
    Then
    \begin{itemize}
        \item[(a)] If $\Sigma$ is stable, then $\Area(\Sigma) \leq \frac{8\pi}{\Lambda_0}$ and
        $diam_\Sigma(\Sigma) \leq \sqrt{\frac{2}{3}} \frac{2\pi}{\sqrt{\Lambda_0}}$
	\item[(b)] If $\Sigma$ has Morse index 1, then
 $\Area(\Sigma) \leq \frac{24\pi}{\Lambda_0}$ and
        $diam_\Sigma(\Sigma) \leq \sqrt{\frac{2}{3}} \frac{4\pi}{\sqrt{\Lambda_0}}$
        \end{itemize}
\end{theorem}

\begin{proof}
    The area estimates can be found in \cite[Proposition A.1]{MaNe12}.
    The diameter estimates can be found in \cite[Corollary 2.4]{LiMa} (see also \cite{SY83}, \cite[Theorem 10.2]{GrLa83}, \cite{Gr2020}, \cite[Proposition 2.2]{LiZh18}).
\end{proof}
\end{comment}

\begin{theorem}
Suppose $(M^3,g)$ is a closed reducible 3-manifold with scalar curvature $R_g\geq 6$.
Then $M$ contains a non-trivial closed geodesic of Morse index at most one and length at most
$\frac{400\pi}{3}$.
\end{theorem}

\begin{proof}
    Since $M$ is reducible there exists an embedded sphere $S$ in $M$
that does not bound a 3-ball. By \cite{MeSiYa82} we can minimize in the isotopy class of $S$ to obtain a stable minimal 2-sphere $\Sigma$. By classical area and diameter estimate for stable minimal surfaces in positive scalar curvature, $e.g,$ \cite[Proposition A.1]{MaNe12} and  \cite[Corollary 2.4]{LiMa}, we have that $\Sigma$ has area $\Area(\Sigma) \leq \frac{4\pi}{3}$ and intrinsic diameter $\diam(\Sigma) \leq \frac{2\pi}{3}$.

By \cite[Main Theorem A]{LNR} there exists a sweepout of $\Sigma$ by 
a 1-parameter family of closed curves $\{F(t) \}_{t \in [0,1]}$,
$F: ([0,1], \partial [0,1]) \rightarrow (\Omega \Sigma, \Omega_0 \Sigma) \subset
(\Omega M, \Omega_0 M)$
of length at most 
\begin{align*}
    L(F(x)) \leq 200d \max\left\{1, \log \Big(\sqrt{\frac{A}{d}}\Big)\right\} \leq \frac{400\pi}{3}.
\end{align*}
Here we use notation $L(\gamma)$ to denote the length of $\gamma$.

Since $\Sigma$ does not bound a 3-ball it follows by Poincare 
Conjecture that $\Sigma$ is not contractible in $M$ and hence
the family $\{F(x)\} $ is not contractible in $(\Omega M, \Omega_0 M)$. Thus, by \cite{bott1982}, it follows that there exists a closed geodesic of index less than or equal to
$1$ of length at most $\frac{400\pi}{3}$.
\end{proof}

If $M$ is a spherical space form, then the problem reduces to the case when $M\cong S^3$,
by mapping a closed geodesic from the universal cover.

\section{Existence of Morse foliation by $L$-short spheres} \label{sec: L-short}

In this section, we will show how to foliate a 3-sphere with 2-spheres that will have a special geometric property we call $L$-shortness. This property will be used in the next section
to construct a two-parameter family of short closed curves.

%In this section we assume that $(M,g)$ is a Riemannian 3-sphere with scalar curvature $R_g\geq \Lambda_0>0$.

\begin{definition}[$L$-Shortness]
%Let $\Sigma \subset M^3$ be an embedded surface.
%We will say that $\Sigma$ is $L$-short in $M$
%if for every smooth map $\gamma: [0,1] \rightarrow \Sigma$ 
%there exists a smooth family of arcs $\{\alpha_t:[0,1] \rightarrow M \}_{t \in [0,1]}$ of length at most $ L$, such that
%$\alpha_t(0) = p \in \Sigma$ and 
%$\alpha_{f(t)}(1) = \gamma(t)$ for some non-decreasing surjective 
%map $f:[0,1] \rightarrow [0,1]$.
%Moreover, if $\gamma(0)=\gamma(1)$, then we can take $\alpha_0 = \alpha_1$

%% It seems that the above definition is not what we need
%% Instead, we need a definition where \gamma is a loop
%% and \alpha starts and ends at a point
%% IMPORTANT: check that all proofs work for this new definition
Let $\Sigma \subset M^3$ be an embedded surface.
We will say that $\Sigma$ is $L$-short in $M$
if for every closed curve $\gamma: [0,1] \rightarrow \Sigma$ 
with $\gamma(0) = \gamma(1)$
there exists a smooth family of arcs $\{\alpha_t:[0,1] \rightarrow M \}_{t \in [0,1]}$ of length at most $ L$, such that
\begin{itemize}
    \item $\alpha_t(0) = \gamma(0)$;
    \item $\alpha_0 = \alpha_1$ is a constant curve;
    \item $\alpha_{t}(1) = \gamma(t)$.
    %$\alpha_{f(t)}(1) = \gamma(t)$ for some non-decreasing surjective map $f:[0,1] \rightarrow [0,1]$.
\end{itemize}
\end{definition}

\begin{remark}
Note that in the definition of $L$-shortness, the interiors of arcs $\alpha_t$ are not required to lie in $\Sigma$.
\end{remark}

We will also need the following notion of tree-foliation, analogous to Definition 3.6 in \cite{LiMa}.

\begin{definition} \label{def: tree-foliation}
 Let $U \subset M^3$ be a subset with $\partial U$ a disjoint
 union of embedded spheres. A family of surfaces $\{\Sigma_x\}_{x \in T}$ is called
 a tree foliation of $U$ if there exist a tree $T$ and continuous functions $f: M \rightarrow T$
 and $g:T \rightarrow \R$, such that 
 \begin{itemize}
     \item  $g \circ f$ is a Morse function and
$T$ is the corresponding Reeb graph (that is, $T = M / \sim$ with $x \sim y$ if
$x$ and $y$ belong to the same connected component of a fiber of $g \circ f$) ;
        \item for each connected component 
$\Sigma$ of  $\partial U$ we have that $\Sigma = f^{-1}(v)$ for a vertex $v$ of degree $1$;
\item  for each edge $E \subset T$ we have that $\{ f^{-1}(t) = \Sigma_t \}_{t \in E^{\circ}}$
        is a smooth family of embedded $2$-spheres.
 \end{itemize}
 \end{definition}

Observe that the following properties of a tree foliation 
follow from the definition:
    \begin{enumerate}
        \item $T$ has vertices of degree $1$ and $3$;
        \item pre-image of a vertex of degree $1$ is a point or a connected component of $\partial U$;
        \item pre-image of a vertex $v$ of degree $3$ satsifes $ f^{-1}(v) = S_1 \cup S_2$,
        where $S_i$ is homeomorphic to a sphere, $S_1 \cap S_2 = \{p\}$
        and $f^{-1}(v) \setminus p$ is smooth and embedded.
        \item For a vertex $v$ of degree $3$ and each edge $E \subset T$  as $t$ converges to $v \in \partial E$
        there is smooth graphical convergence of $\Sigma_t$ to a subset of $f^{-1}(v)$ away from the singular point 
        $p \in f^{-1}(v)$.
    \end{enumerate}

%\begin{remark}
%For example, the map into a Reeb graph corresponding to a Morse function on $M$
%is a tree foliation.
%\end{remark}

The main result of this section is the following theorem.

\begin{theorem} \label{thm: L-short foliation}
    Let $M^3$ be a 3-sphere with scalar curvature $R\geq 6$.
    There exists a tree foliation $\{\Sigma_t \}$ of $M$,
    such that each $\Sigma_t$ is $L$-short
    for an 
    $L<4500$.
    %$L < 3000$.
\end{theorem}

%Add a bit about the proof strategy of Theorem \ref{thm: L-short foliation}.

\subsection{Area and diameter control implies $L$-shortness}
Using the following result from \cite{LNR}, we show that an area and diameter upper bound for a 2-sphere implies it must be $L$-short.

\begin{theorem} \label{LNR}
    Suppose $\Sigma$ is a Riemannian 2-sphere
    or a $2$-disc
    of area $\Area (\Sigma) \leq A$ and diameter $\diam(\Sigma) \leq d$ and $\gamma_0, \gamma_1$ are two paths
    connecting points $p, q \in \Sigma$. Then there exists a homotopy $\gamma_t$
    from $\gamma_0$ to $\gamma_1$ with length
 $L(\gamma_t) \leq 2(L(\gamma_1)+ L(\gamma_0))+ 200d \max\{1, \log(\frac{\sqrt{A}}{d})\}$.
\end{theorem}

\begin{proof}
    The result follows if both $\gamma_0$ and $\gamma_1$ can be homotoped 
    to the same length minimizing geodesic between $p$ and $q$ through curves with the
    the desired length bound. Hence,
    without any loss of generality, we may assume that $\gamma_1$ is a length-minimizing geodesic.
    
    Subdivide $\gamma_0$ into $N$ small arcs $\{ \alpha_i \}_{i=0}^N$ with endpoints $\partial \alpha_i = \{x_i, x_{i+1} \}$ of length less than $ \varepsilon< \text{\rm convrad}(\Sigma)$. After a small
    perturbation we can assume that each $\alpha_i$ is a geodesic arc. Let $\tau_i$ denote a length minimizing geodesic
    from $p$ to $x_i$. Consider a sequence of curves $c_i = \tau_i \cup \bigcup_{j \geq i} \alpha_j$.
   Since $\alpha_i$, $\tau_i$ and $\tau_{i+1}$ are minimizing geodesics, they do not intersect except at the endpoints.
   Hence, they bound a disc $D_i$.
   By \cite[Theorem 1.2]{LNR}, there exists 
   a contraction of $\tau_i \cup \alpha_i \cup\tau_{i+1}$
   to a point through based loops with length increase
   bounded by $L(\tau_i \cup \alpha_i )+ L(\tau_{i+1})+ 200d \max\{1, \log(\frac{\sqrt{A}}{d})\}$. 
   By Lemma \ref{pathhomotopy} there exists a path homotopy
   from $\tau_i \cup \alpha_i$ to $\tau_{i+1}$
   through curves of length at most 
   $$\min \{L(\tau_i \cup \alpha_i ), L(\tau_{i+1}) \} + L(\tau_i \cup \alpha_i )+ L(\tau_{i+1})+ 200d \max\left\{1, \log(\frac{\sqrt{A}}{d})\right\}.$$ 
It   follows that we have a sequence of homotopies from
    $c_i$ to $c_{i+1}$, starting with $\gamma_0=c_0$ and ending with $c_N = \gamma_1$, satisfying the desired length bound.
\end{proof}

As a consequence of Theorem \ref{LNR} we have the following lemma:

\begin{lemma} \label{lemma: small spheres are L-short}
Suppose $\Sigma$ is a 2-sphere or a 2-disc in $M$
with intrinsic diameter $d$ and area $A$.
Then for every $\varepsilon>0$ $\Sigma$ is $L$-short for $L \leq 4d + 200d \max\{1, \log(\frac{\sqrt{A}}{d})\} + \varepsilon$.
%    Suppose $\Sigma$ is a 2-sphere with  $\Area(\Sigma)\leq A$ and diameter $\diam(\Sigma)\leq d$.
%    Let $\gamma: [0,1] \rightarrow M$ be a curve in $\Sigma$. Then for any $p \in \Sigma$ there exists 
%    a family of arcs $\{\alpha_t\}_{t\in [0,1]}$ of length at most
%    $L \leq 204d \max\{1, log(\frac{\sqrt{A}}{d})\}$, so that $\alpha_t(0)=p$
%    and $\alpha_t(1) =\gamma(t)$. Moreover, if $\gamma$ is a loop, then
%    we can choose $\alpha_0=\alpha_1 = \{ p \}$. In particular, for any isometric embedding of $\Sigma$
%    in $M$, $\Sigma$ is $L$-short. If $\Sigma$ is a 2-disc with boundary length
 %   $\leq L'$, then $\Sigma$ is $L$-short for $L \leq L' + 204d \max\{1, log(\frac{\sqrt{A}}{d})\}$.
\end{lemma}

\begin{proof}
   Let $\gamma$ be a closed curve in $\Sigma$ and $p \in \gamma$. We will construct a family of arcs connecting $p$ to the points of $\gamma$ and satsifying the desired length bound.
   Subdivide $\gamma$ into small arcs $\gamma_i$
   by a sequence of pointa $\{p=a_0, ..., a_N =p \}$
   with $\partial \gamma_i = \{a_{i-1}, a_i\} $,
   such that 
each $\gamma_i$ has length bounded by $\delta < \frac{\varepsilon}{10}$. Consider a sequence of  geodesics $\{ \alpha_i \}_{i=1}^{N-1}$ connecting $p$ to $a_i$ and that are length minimizing in $\Sigma$. Note we have $L(\alpha^i)\leq d$. 
By Theorem \ref{LNR} there exists a homotopy $\alpha^i_t$ from 
$\alpha_{i-1} $ to $\gamma_i \cup \alpha_i$
through curves of length bounded by
\begin{align*}
 L( \alpha^i_t) & \leq 2(L( \alpha_{i-1} \cup \gamma_i) + L(\alpha_i)) + 200d \max\left\{1, \log(\frac{\sqrt{A}}{d})\right\}\\
 & \leq 4d + 2\delta + 200d \max\left\{1, \log(\frac{\sqrt{A}}{d})\right\}
\end{align*}
  By adding a family of subarcs of $\gamma_i$ connecting the endpoint $a_{i-1}$ of $\alpha^i_t$ to the points of $\gamma_i$ we obtain
  a family of curves connecting $p$ to the points 
  of $\gamma_i$ that interpolated between $\alpha^{i-1}$
  and $\alpha^i$. Together these homotopies give us the
  desired homotopy $\alpha_t$.
\end{proof}

\subsection{Surgery and $L$-shortness} To construct our desired foliation, we will need to
perform certain cut-and-paste surgeries on $2$-spheres with controlled area and diameter. 
Unfortunately, these operations do not necessarily preserve the bound on intrinsic diameter.
However, the next proposition shows how $L$-shortness can be preserved in appropriate cases.

\begin{proposition} \label{L-short in a ball}
    Let $\eps< \text{\rm convrad}(M)$\footnote{Throughout the paper, $\text{\rm convrad}(M)$ denotes the convexity radius of $M$.}. Suppose $\Sigma_1$ is an embedded sphere
    with $\diam(\Sigma_1) \leq l_0$ and $\Area(\Sigma_1) \leq l_0^2$.
    Let $B$ denote a closed ball of radius $\eps$.
    Suppose $\Sigma_2 \subset M$ is an embedded sphere, such that 
     $S = \Sigma_2 \cap B$  
     satisfies the following: 
\begin{enumerate}
    \item $\Sigma_2 \setminus \Sigma_1 \subset S$;
    \item For each connected component $D$ of $S$
    we have $\partial D \subset \partial D'$,
    where $D'$ is a connected component of $\Sigma_1 \cap B$.
\end{enumerate}
   Then $\Sigma_2$ is $L$-short for $L \leq 204 l_0 +3\varepsilon$.
\end{proposition}

\begin{figure}[h!]
\centering	
\includegraphics[width=75mm,scale=0.4]{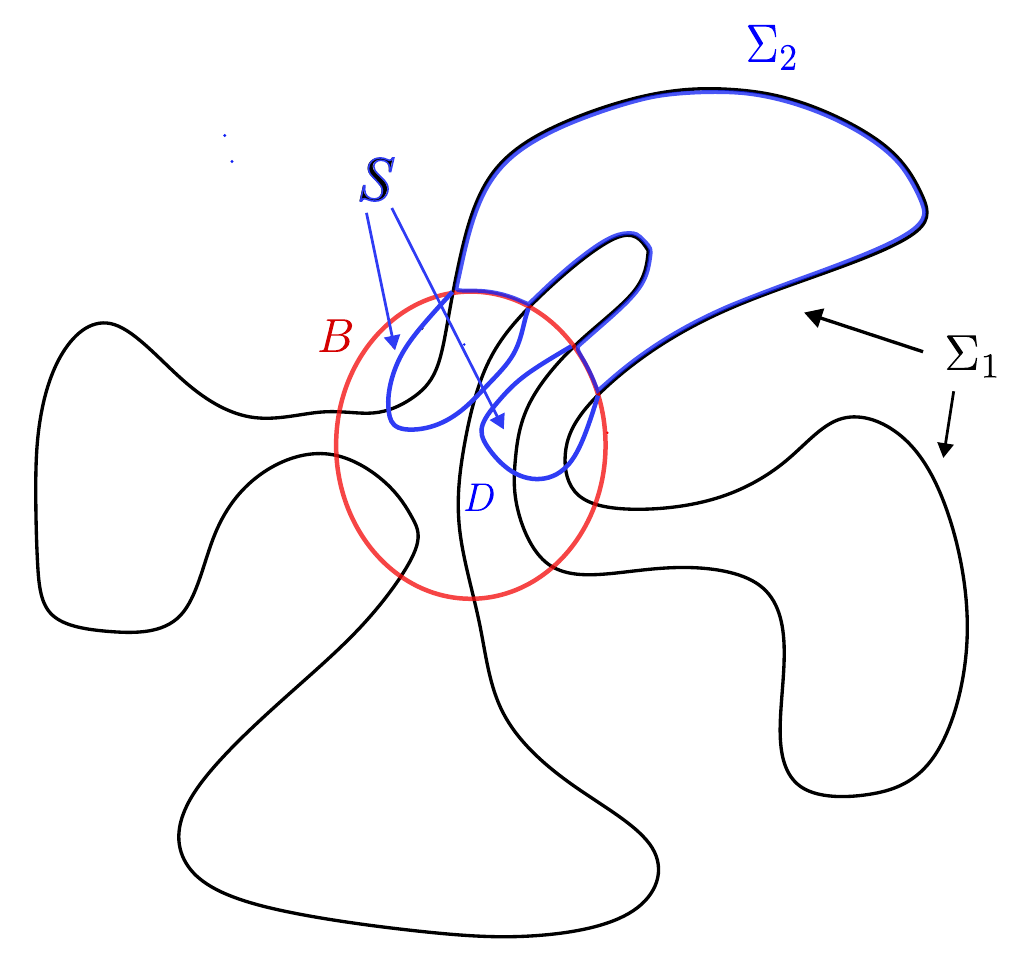}
\caption{$L$-short $\Sigma_2$ after cut-and-paste surgeries inside $B$}
\end{figure}
\begin{proof}
We will prove a slightly stronger result then $L$-shortness.
We will show that given any point $p \in \Sigma_2$ and a curve $\gamma: [0,1] \rightarrow \Sigma_2$, there exists
a family of arcs $\alpha_t$ connecting $p$ to $\gamma(t)$ and, moreover, if $\gamma(0) =\gamma(1) = \{p\}$,
then we can take $\alpha_0 = \alpha_1=\{p\}$. 
We will assume that $p  \in \Sigma_1 \cap \Sigma_2$; the case $p \in \Sigma_2 \setminus \Sigma_1 \subset B$
is a straightforward modification of the argument below. 

Let $\gamma \subset \Sigma_2$ be a path.  Subdivide $\gamma$ into segments $\gamma_i$ of length $<\delta<\eps$ and let $\{ a_i \}_{i=0}^N$ denote the endpoints
  of $\{ \gamma_i \}$, $\partial \gamma_i = \{a_i, a_{i+1} \}$, with
  $a_0 = \gamma(0), a_N= \gamma(1)$. 
  Without any loss of generality we can assume that $\gamma$ 
  intersects $\partial S$ transversally and that
  $\gamma \cap \partial S \subset \{ a_i\}$. We now construct
  a collection of arcs $\alpha_i$ from $p$ to $a_i$. We consider the following two cases:
  \begin{enumerate}
      \item $a_i$ is contained
  in the interior of $\Sigma_2 \setminus S \subset \Sigma_1$; we connect $a_i$ to $p$ by a curve $\alpha_i \subset \Sigma_1$ of length $\leq \diam(\Sigma_1) \leq l_0$ in $\Sigma_1$; \label{case1}
  \item  $a_i$ is contained in a connected component $D$ of $S$, see Figure \ref{Fig16}. Let 
  $b_i$ be a point in $ D \cap \partial B \in \Sigma_1$. We connect $a_i$ to 
point $b_i$ by a geodesic (in the ambient metric) $\tau_i \subset B$ of length at most $\varepsilon$; we connect $b_i$ to $p$ by a curve $\beta_i$ of length $\leq \diam(\Sigma_1) \leq l_0$ in $\Sigma_1$; we set $\alpha_i = \tau_i \cup \beta_i$. \label{case2}
%\item  $a_i$ is contained in a connected component $D$ of $\Sigma_2 \setminus  ( B \cup \Sigma_1)$;
%we connect  $a_i$ to a
%point $b_i \in \partial D  \in \Sigma_1$ by a geodesic $\alpha_i \subset D$ of length at most $l_0$; we %connect $b_i$ to $p$ by a curve $\beta_i$ of length $\leq \diam(\Sigma_1) \leq l_0$; we set $\tau_i = %\alpha_i \cup \beta_i$; \label{case3}
  \end{enumerate}

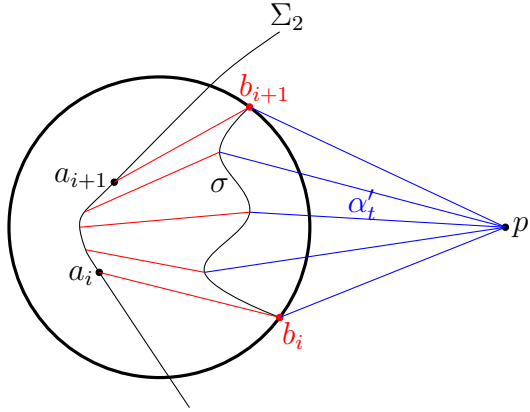
\begin{figure}[h!] 
    \centering
\begin{tikzpicture}[scale=0.2]

\draw [very thick] (-10,0) circle (10);

\draw plot[very thick, smooth ,tension=.5] coordinates { (-2,13) (-6,10) (-13,3) (-15,1) (-15.3,0) (-14.9,-1.5)    (-14,-3)  (-8,-12)};

\node at (-1.5,14) {$\Sigma_2$}; 

\draw plot[very thick, green, smooth ,tension=.65] coordinates { (-4,8) (-6,5) (-4,1) (-7,-3) (-2,-6)};

\node at (-6,2.8)  {$\sigma$};
\node[red] at (-3,9.2) {$b_{i+1}$};
\node[red] at (-1.1,-7.2)  {$b_i$};

\node at (-15.2,3.2) {$a_{i+1}$};
\node at (-15.2,-3.2)  {$a_i$};
\node[blue] at (3.5,1.5) {$\alpha'_t$};

\node at (14,0) {$p$};

\filldraw[red] (-4,8)
circle[radius=6pt];
\filldraw[red] (-2,-6)
circle[radius=6pt];
\filldraw[black] (-13,3)
circle[radius=6pt];
\filldraw[black] (-14,-3)
circle[radius=6pt];

\filldraw[black] (13,0)
circle[radius=6pt];

\draw [blue] (13,0) --(-2,-6);
\draw [blue] (13,0) --(-4,8);
\draw [blue] (13,0) -- (-6,5);
\draw [blue] (13,0) -- (-4,1);
\draw [blue] (13,0) --  (-7,-3);
\draw [red] (-4,8)--(-13,3);
\draw [red] (-6,5) -- (-15,1) ;
\draw [red] (-2,-6) --  (-14,-3) ;
\draw [red] (-4,1) --   (-15.3,0);
\draw [red] (-7,-3) --  (-14.9,-1.5)  ;

\end{tikzpicture}

\begin{comment}
  \draw[name path=curve 1] plot[smooth cycle,tension=.65] coordinates {
    (-5.5,10) (-10,12.5) (-13,11.5) (-17,8) (-17.5,5) (-16.5,3) (-13,1) (-8,3) (-4.5,4.5) (-2,3.5) (-4.5,0) (-9,-4) (-15,-7.5) (-10,-10.5) (-4.5,-9) (-3,-5)  (-2,-1) (0,3) (2.5,8) (8,10) (11,8) (4,5) (2,3) (.5, 0) (1,-5) (-.5,-9) (4,-11.5) (11,-10) (13,-6) (10, -2) (4.5,-3) (2.5,-1.5) (3,2) (5,4) (11,5) (16.5,2.5) (18.5,7) (17,11) (11,13.5) (7,13) (4,11) (2,9) (0,6) (-4,6)  };

\draw [name path=curve 2, red, very thick] (2,5) circle (6);

\fill [name intersections={of=curve 1 and curve 2, by={1,2,3,4,5,6,7,8,9,10}}]

\draw (9) -- (5);

 %\fill [name intersections={of=curve 1 and curve 2, name=i, total=\t}][red, opacity=0.5, every node/.style={above left, black, opacity=1}]\foreach \s in {1,...,\t}{(i-\s) circle (2pt) node {\footnotesize\s}};

%\draw plot[very thick, blue, smooth cycle,tension=.65] coordinates {(8) (11,5) (16.5,2.5) (18.5,7) (17,11) (11,13.5) (7,13) (4,11) (9)  };

\end{comment}
\caption{Case 2: $a_i$ is contained in a connected component $D$ of $S$}
\label{Fig16}
 \end{figure}

In case (\ref{case2}), if $a_i \in \partial B$, then we set $b_i=a_i$.
  
We now claim that there exists a homotopy of arcs with fixed endpoints $p$ and $a_{i+1}$
interpolating between $\alpha_i \cup \gamma_i$ and $\alpha_{i+1}$ through curves of length
$\leq 204 l_0 +2 \delta$. 
Putting these deformations together we obtain a homotopy from $\alpha_0$ to $\alpha_1$
with desired properties.

If $a_i, a_{i+1} \in \Sigma_1$, then existence of desired homotopy follows from Theorem \ref{LNR}.

Suppose $a_i, a_{i+1}$ lie in a connected component $D$ of $S =\Sigma_2 \cap B$ and 
let $D'$ denote a connected component of $\Sigma_1 \cap B$ with $\partial D \subset \partial D'$.
Let $\sigma \subset D'$ be a curve connecting $b_i$ to $b_{i+1}$. Note that
we don't have a bound on the length of $\sigma$.
By Lemma \ref{lemma: small spheres are L-short} there exists a family of 
arcs $\alpha_t' \subset \Sigma_1$ of controlled length interpolating between
$\beta_i$ and $\beta_{i+1}$ with endpoints in $\sigma$. Since $\sigma$
lies inside the ball $B$ there exists a family of arcs connecting
$\sigma$ to $\alpha_{i+1}$ through curves of length $\leq 2 \varepsilon$.
This gives the desired family of arcs $\alpha_t$.
\end{proof}

\begin{definition}
    Let  $V \subset M$ be a domain with smooth boundary
    and $\Sigma \subset \partial V$ be an embedded 2-sphere.
    Let $\{ D_i \}_{i=1}^k $
    be a collection of disjoint embedded 2-discs in $V$, satisfying:
    \begin{itemize}
        \item $Int(D_i) \subset M \setminus \Sigma$;
        \item  $\partial D \subset \Sigma$ is a disjoint union of embedded closed
         curves.
    \end{itemize}
 %Let $U_i$ denote $\varepsilon$-neighbourhood of $D_i$ in $V$.
 A family of embedded surfaces $\{ \Sigma_t \}$
 is called an $\varepsilon$-neck-pinching along $\{ D_i \}$
 if it is obtained from $\Sigma$ by deforming the small cylinders $N_{\varepsilon}(\partial D_i) \cap \Sigma$
 along $D_i$ until they pinch. 

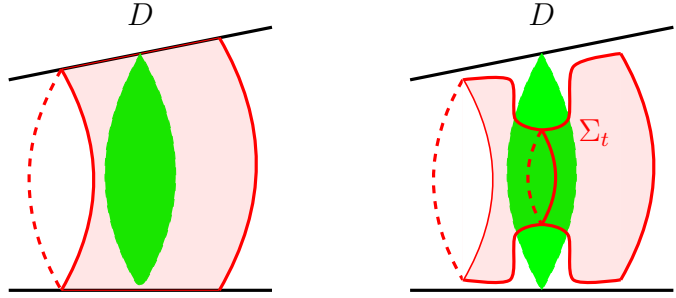
\begin{figure}[h!] 
    \centering
\begin{center}
\begin{minipage}{.2\textwidth}
\begin{tikzpicture}[scale=0.35] %[x={10.0pt},y={10.0pt}]
\draw[very thick] (4,9) -- (9,10) -- (14,11);
\draw[very thick] (4,1) -- (9,1) -- (14,1); 

\draw[very thick, red,dashed] (6,9.4) to [bend right =30] (6,1);
\draw[very thick, red]  (6,9.4) to [bend left =30] (6,1) ;

\draw[red]  (6,9.4) -- (12,10.6);
\draw[red]  (6,1)  -- (12,1);

%\draw[very thick,red] (12,10.6) to [bend right =30] (12,1);
\draw[very thick, red]  (12,10.6) to [bend left =30] (12,1) ;

\filldraw[green,dashed, rounded corners] (9,10) to [bend right =30] (9,1) to [bend right =30] (9,10)  ;

\node at (9,11.5)  {$D$};

\begin{scope}
\clip  (6,9.4) to [bend left =30] (6,1)  to (12,1) to [bend right =30]  (12,10.6)  to (6,9.4);

%\draw[red] (4,2) to (17,2);
%\draw[red] (4,3) to (17,3.5);
%\draw[red] (4,4) to (17,5);
%\draw[red] (4,5) to (17,6.5);
%\draw[red] (4,6) to (17,8);
%\draw[red] (4,7) to (17,9.5);
%\draw[red] (4,8) to (17,11);

\fill[red, opacity=0.1] (6,9.4) to [bend left =30] (6,1) to (12,1) to [bend right =30] (12,10.6) to (6,9.4);
\end{scope} 
\end{tikzpicture}
\end{minipage}
%  no empty line here
\begin{minipage}{.2\textwidth}
	\begin{tikzpicture}
	%\draw[white,fill=white] (-1,0) circle (0.1);
	%\draw[very thick,->]   (-1,0)   -- (1,0);
	%\draw[white,fill=white] (3,0) circle (0.1);
 \end{tikzpicture}
\end{minipage}
%  no empty line here 
\begin{minipage}{.2\textwidth}
\begin{tikzpicture}[scale=0.35] %[x={10.0pt},y={10.0pt}]
\draw[very thick] (4,9) -- (9,10) -- (14,11);
\draw[very thick] (4,1) -- (9,1) -- (14,1);

\draw[very thick, red,dashed] (6,9) to [bend right =30] (6,1.4);
\draw[very thick, red]  (6,9) to [bend left =30] (6,1.4) ;

%\draw[red]  (6,9.4) -- (12,10.6);
%\draw[red]  (6,1)  -- (12,1);

%\draw[very thick,red] (12,10.6) to [bend right =30] (12,1);
\draw[very thick, red]  (12,10) to [bend left =30] (12,1.4) ;

\filldraw[green,dashed] (9,10) to [bend right =30] (9,1) to [bend right =30] (9,10)  ;

\node at (9,11.5)  {$D$};

\draw[very thick, red] plot [ smooth, tension=.5]   coordinates{(6,9) (7.8,9) (8,7.5) (9,7.1)  (10,7.5) (10.2,9.5) (12,10) };
\node[red] at (11,7)  {$\Sigma_t$};
\draw[very thick, red] plot [ smooth, tension=.5]   coordinates{(6,1.4) (7.8,1.4) (8,3.1) (9,3.5) (10,3.1) (10.2,1.4) (12,1.4)};

\draw[dashed] [very thick,red] (9,7.1) to [bend right =30] (9,3.5);
\draw [very thick,red] (9,7.1) to [bend left =30] (9,3.5);

\path[name path=1]   plot [ smooth, tension=.5]   coordinates{ (6,9) (7.8,9) (8,7.5) (9,7.1)  (10,7.5) (10.2,9.5) (12,10) };
\path[name path=2]  plot [ smooth, tension=.5]   coordinates{(6,1.4) (7.8,1.4) (8,3.1) (9,3.5) (10,3.1) (10.2,1.4) (12,1.4)};

\path[name path=3]  (12,10) to [bend left =30] (12,1.4);
\path[name path=4] (6,9) to [bend left =30] (6,1.4);
\path[name path=5]  (12,10) to  (12,1.4);
\path[name path=6] (6,9) to (6,1.4);

\tikzfillbetween[of=1 and 2]{red, opacity=0.1};
 \tikzfillbetween[of=3 and 5]{red, opacity=0.1};
\tikzfillbetween[of=4 and 6]{white};

\begin{scope}
\clip  (6,9.4) to [bend left =30] (6,1)  to (12,1) to [bend right =30]  (12,10.6)  to (6,9.4);

%\draw[red] (4,2) to (17,2);
%\draw[red] (4,3) to (17,3.5);
%\draw[red] (4,4) to (17,5);
%\draw[red] (4,5) to (17,6.5);
%\draw[red] (4,6) to (17,8);
%\draw[red] (4,7) to (17,9.5);
%\draw[red] (4,8) to (17,11);

%\fill[red, opacity=0.1] (6,9.4) to [bend left =30] (6,1) to (12,1) to [bend right =30] (12,10.6) to (6,9.4);

\end{scope}

\end{tikzpicture}

\end{minipage}

\end{center}
\caption{$\varepsilon$-neck-pinching along a disk $D$}
 \end{figure}
 
More precisely, given an $\varepsilon>0$, we realize $\varepsilon$-neck-pinching as a tree-foliation $\{ \Sigma_t \}_{t \in T}$ of a region $V'$,
 $N_{\varepsilon/2}(\Sigma \cup \bigcup D_i)\cap V \subset V' \subset N_{\varepsilon}(\Sigma \cup \bigcup D_i)\cap V$. The set of vertices of $T$ can be decomposed $\mathcal{V}=\{v_j\}_{j=0}^{k+1} \cup \{w_l \}_{l=1}^k$, such that:
 \begin{enumerate}
     \item Vertices $\{v_j\}_{j=0}^k$ have degree $1$, $\Sigma_{v_0}=\Sigma$
     and $\bigsqcup_{j=1}^{k+1} \Sigma_{v_j} = \partial V' \setminus \Sigma$;
     \item Vertices $\{w_l \}_{l=1}^k$ have degree $3$.
 \end{enumerate}
 Moreover, if $\Sigma$ is minimal or mean convex, and $D_i$ are minimal embeddings,
 then $\partial V' \setminus \Sigma$ is a collection of embedded spheres with
 mean curvature pointing inside $V \setminus V'$.
\end{definition}

The next proposition shows how $L$-shortness behaves under $\varepsilon$-neck-pinching.

\begin{proposition} \label{prop: neck-pinching L-short}
    Suppose $\Sigma$ is an $L$-short embedded sphere and $\{ D_i\}_{i=1}^k$
 a collection of disjoint embedded discs with $\partial D_i \subset \Sigma$, satisfying:
\begin{enumerate}
    \item[(a)] $\length(\partial D_i)<l_1$,
    \item[(b)] $\diam(D_i) \leq l_2$,
    \item[(c)] $\Area(D_i) \leq l_2^2$,

\end{enumerate}
Suppose $\{ \Sigma_t \}$ is an $\varepsilon$-neck-pinching family along $\{ D_i \}$.
    Then $\Sigma_t$ is $L'$-short for $L' \leq L + l_1 + 204 l_2+ 4\varepsilon$.
\end{proposition}

\begin{proof}
    Let $\gamma$ be a curve on $\Sigma_t$. We will construct 
    a family of short arcs $\alpha_t$ connecting points of $\gamma$ to a point $p \in \gamma$.
    
    Let $S_t = \Sigma \cap N_\varepsilon(\Sigma)$. 
    We will fist suppose that $p \in S_t$.
    After a small perturbation, we may assume that $\gamma$ intersects 
    $\partial S_t$ transversally and let $A^{i}_t$ denote the connected components
    of $\Sigma_t \setminus S_t$ that lies in the small neighbourhoods $N_\varepsilon(D_i)$.

     For each connected component $a_j \subset A^{i}_t$
    of $\gamma \setminus S_t$, let $b_j$ be a curve in $\Sigma \cap N_{\varepsilon}(\partial D_i)$ connecting the endpoints $\partial a_j$.  It follows from Lemma \ref{lemma: small spheres are L-short} that there exists a continuous family of arcs $\alpha^j_s$
    connecting $a_j(s)$ to $b_j(s)$, starting and ending on a constant map, and of length bounded by $l_1 + 204 l_2  + 2 \varepsilon$.
    (Even though $a_j$ do not lie on the discs $D_j$, they  belong to an $\varepsilon$-neighborhood of 
    $D_j$, so, the result follows by projecting $a_j$ onto $D_j$ and connecting 
    points of the projection $a_j'$ to $a_j$ by arcs of length $<\varepsilon$.)

    Let $\gamma' \subset S'$ be a curve obtained from $\gamma$ by replacing each $a_j$ with $b_j$. Let
    $\gamma''$ denote the curve obtained by projecting $\gamma'$ to $\Sigma$. Since $\Sigma$ is $L$-short,  
    there exists a family of arcs $\alpha_t'$ connecting $p$ to $\alpha_t'(1) = \gamma''(t)$
    of length bounded by $L + \varepsilon$.
    Adding a small arc to $\alpha_t'(1)$ in the tubular neighbourhood of $\Sigma$ to connect it to $\gamma'(t)$ and then adding arcs $\alpha^j$
    gives the desired family of arcs with length bounded 
    by $l_1 + 204l_2 + 4 \varepsilon$.
\end{proof}

\begin{remark}\label{remalocalpinch}
We will also need the following local version of Proposition \ref{prop: neck-pinching L-short}:
Given $(M^3,g)$, let $r>0$ be  real number smaller than the convexity radius
of $M$.
%and such that every ball of radius $r$ in $M$ is $2$-biLipschitz diffeomorphic to a Euclidean ball of the same radius.  
Suppose $\Sigma$ is an $L$-short 2-sphere and $D$ is an embedded disk contained inside of a ball $B_r\cap\Sigma$. Then, there exists an $\varepsilon$-neck-pinching family $\{\Sigma_t\}$ along $D$ such that $\Sigma_t$ is $L'$-short for $L'\leq L+2r$. The proof follows by a straightforward modification of the proof of Proposition \ref{prop: neck-pinching L-short}.
\end{remark}

\subsection{Discrete MCF with surgery}
We don't know how to control $L$-shortness 
of surfaces in Mean Curvature Flow.
To overcome this issue we describe a combinatorial
version of the flow. In this version, we repeatedly replace
portions of the surface in a small ball by a certain minimizer in the isotopy
class until the surface converges to a stable minimal surface.

We need the following slight modification of the notion of geometrically prime regions from \cite{LiMa}. 

\begin{definition} \label{def: prime}
A 3-manifold $V$ with non-empty boundary is
geometrically prime if 
\begin{enumerate}
    \item $V$ is diffeomorphic to $S^3$ with some 3-balls removed;
    \item there are no closed embedded minimal surfaces
    in the interior of $V$;
    \item there exists a closed connected component $\Sigma$ 
(that we'll call ``large'') of $\partial V$ that is either mean convex or minimal of Morse index $1$;
    \item $\partial V\setminus \Sigma $ is either empty or a disjoint union 
    of stable minimal 2-spheres.
\end{enumerate}

\end{definition}

\begin{definition} \label{def: isotopy minimizer}
    Let $U$ be an open set with mean convex boundary and $\Sigma \subset cl(U)$ be a surface with $\partial \Sigma \subset \partial U$.
    By the Theorem of W. Meeks, L. Simon and S.T. Yau \cite{MeSiYa82} (see also  \cite[Theorem 7.3]{CD})  
there exists a minimal surface $S \subset U$
    with $\partial S = \partial \Sigma$ obtained by minimizing area among surfaces isotopic
    to $\Sigma$ with fixed boundary and possibly with some necks pinched. We call such $S$ an isotopy minimizer for $\Sigma$ in $U$.
    Let $\mathcal{A}(\Sigma, U)$ denote the area of an isotopy minimizer for $\Sigma$ in $U$.
\end{definition}

The next lemma shows how to foliate the region between a surface $\Sigma$ and a new surface $\hat \Sigma$ obtained from $\Sigma$ by replacing with an isotopy minimizer in a small ball. 

\begin{lemma}\label{lemma:folsmall}
    Let $U$ be a geometrically prime region in $(M,g)$ and
    $\Sigma$ its large connected component of $\partial U$.
    Let $r, \varepsilon \in(0,\text{\rm convrad}(M)/2]$,  $\varepsilon<\frac{1}{100}$ and $p\in\Sigma$.
    Suppose $\partial B_r(p)$ interesects $\Sigma$ transversally and let $S$ be an isotopy minimizer for $\Sigma \cap B_r(p)$ in $U \cap B_r(p)$.
    Then there exist a mean convex surface $\hat\Sigma$ and subset $V$ of $U$ lying between $\Sigma$ and $\hat \Sigma$,
    such that $V$ admits a Morse function
    $g:V\longrightarrow[0,1]$ with the following properties:
 \begin{enumerate}
     \item[\rm (1)] $\Sigma = g^{-1}(0)$;
     \item[\rm (2)] $\hat\Sigma = g^{-1}(1)$, and it  
     satisfies the area estimate: $$\Area(\hat\Sigma) \leq \Area(\Sigma \setminus B_r(p)) +  \mathcal{A}(\Sigma, B_r(p) \cap U); $$
     \item[\rm (3)] $g^{-1}(t) \setminus B_{r +2\eps}(p)$ is a family of graphical surfaces over $\Sigma$ with $C^{2,\alpha}$-norm less than $2\varepsilon$.
 \end{enumerate}
 \end{lemma}

\begin{proof}
%The proof is based on induction on the number of connected components of \( \partial \Sigma \cap B_r(p) \).

Consider first the case where \( \partial (\Sigma \cap B_r(p)) \) is isotopic to the minimizer \( S \cap B_r(p) \) through isotopies inside the ball \( B_r(p) \) with fixed boundary, as shown in Figure \ref{Fig20}.

\begin{figure}[h!]
    \centering
    \begin{tikzpicture}[scale=0.25] 
%\draw[help lines,step=1] (-10,-10) grid (10,10);

\draw[very thick](0,0) circle (10);

\draw[very thick,red ] (-10,0) to [bend left=15] (10,0);

\draw[very thick] plot [ smooth, tension=.8]   coordinates{ (-14,-4) (-10,0) (-6,5) (0,2) (6,4)(10,0) (14, -4)};

\draw[very thick,blue] plot [ smooth, tension=.4]   coordinates{ (-14,-5) (-9.5,-1) (0,0) (9.5,-1) (14, -5)};

\node at (0,4) {$\Sigma$};

\node[red] at (-6,2) {$S$};

\node[blue] at (-6,-2) {$\hat S$};

\node at (2,-8) {$B_{r_0}(p)$};

\end{tikzpicture}
    \caption{$S \cap B_r(p)$ and \( \Sigma \cap B_r(p) \) are isotopic.}
    \label{Fig20}
\end{figure}
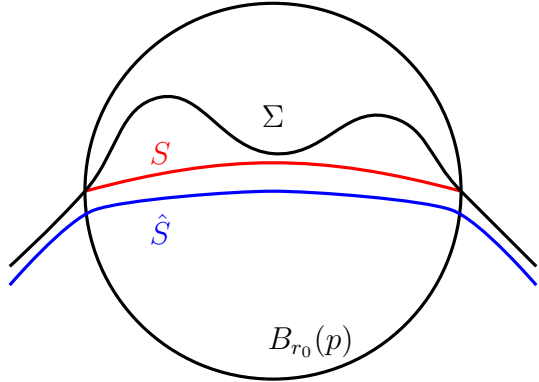

Since \( S \cup (\Sigma \setminus B_r(p)) \) is a 2-sphere with corners (along \( \partial B_r(p) \)), \( \Sigma \) is mean convex, and \( S \cap B_{r_0}(p) \) is minimal, we can flow \( S \cup (\Sigma \setminus B_r(p)) \) by mean curvature flow for a small time \( t_0 > 0 \), obtaining a family of mean convex surfaces \( \Si_t \) terminating at \( \hat \Sigma \).
Let $\hat S$ denote $\hat \Sigma \cap B_r(p)$.
By interior estimates for mean curvature flow, $e.g.$~\cite{EcHu91}[Theorem 4.2], if we choose \( t_0 = t_0(\varepsilon) \) sufficiently small, \( \hat \Sigma\) will satisfy a bound on its $C^{2, \alpha}$-norm as in property (3). Additionally, since \( S \cup (\Sigma \setminus B_r(p)) \) has mean curvature pointing away from \( \Sigma \), the monotonicity of area along the flow ensures that \( \hat \Sigma\) satisfies the inequality in (2). The desired Morse function \( g: V \longrightarrow [0,1] \) can then be constructed by noting that \( V \) is diffeomorphic to \( \hat \Sigma \times [0,1] \), and one can patch a foliation of \( V \cap B_r(p) \) with the foliation of \( V \setminus B_r(p) \) by the surfaces \( \Si_t \) outside of \( B_r(p) \).

By the  Definition \ref{def: isotopy minimizer} of isotopy minimizers we have that the topology of $S$ can defer from the topology of $\Sigma \cap B_{r_0}(p)$ only 
by finitely many neck-pinches. If follows that $n_c(S)$, the number of connected components of $S$, is greater than or equal to $n_c(\Sigma \cap B_{r_0}(p))$, the number of connected components of  $\Sigma \cap B_{r_0}(p))$. Moreover, $n_c(S) = n_c(\Sigma \cap B_{r_0}(p))$ only if $S$
and $\Sigma \cap B_{r_0}(p)$ are isotopic in $B_{r_0}(p)$. The proof now proceeds by induction on $k = n_c(S)-n_c(\Sigma \cap B_{r_0}(p))$.

%For the induction hypothesis, suppose the result is true whenever \( \Sigma \cap B_r \) has at most \( k \) boundary components, where \( k \in \mathbb{N} \) is fixed. Assume there is a certain choice of \( U \), \( p \), and \( r \), where \( \Sigma \cap B_r \) has \( k+1 \) boundary components. Let \( S \) be the isotopic minimizer of \( \Sigma \cap B_r(p) \) inside \( B_r(p) \). In case there exists an isotopy of \( B_r(p) \) between \( S \) and \( \Sigma \cap B_r \), it is not hard to see that we can proceed as before. 

Assume we have proved the Lemma for a fixed value of $k \geq 0$.
Suppose the minimizer $S$
had $n_c(\Sigma \cap B_{r_0}(p))+k+1$ connected components  
(Figure \ref{Fig19} illustrates the case 
$n_c(S) =2$, $n_c(\Sigma \cap B_{r_0}(p))$).

%We then need to consider the possibility that the minimizer \( S \) is not smoothly isotopic to the original surface \( \Sigma \cap B_r(p) \), as shown in Figure \ref{Fig19}. This can occur if \( \Sigma \cap B_r(p) \) has several boundary components and disconnecting them reduces the area.

\begin{figure}[h!]
    \centering
    \begin{tikzpicture}[scale=0.25] 
%\draw[help lines,step=1] (-10,-10) grid (10,10);

\draw[very thick](0,0) circle (10);
\draw[thin] (-10,0) to [bend right=25] (10,0);
\draw[thin,dashed] (-10,0) to [bend left=25] (10,0);

\draw[very thick] (-5.2,4.3) circle [x radius=2cm, y radius=1cm, rotate=40];
\draw[very thick] (6,3) circle [x radius=2cm, y radius=4cm, rotate=40];
\draw[very thick] (-5,-5) circle [x radius=1cm, y radius=1.5cm, rotate=40];

\draw[very thick] (-4.1,5.8) to  (3.9,6.25);
\draw[very thick, red] plot [ smooth, tension=.8]   coordinates{ (-6.75,3) (-5,2.4) (-3,5)   (-4.1,5.8)};

%\draw[red, very thick] (-2.5,5.4) to [bend right=20]  ( 2.4,5);
\draw[very thick] (-6.75,3) to [bend left=25] (-6.1,-4);
%\draw[red,very thick] (-6.4,1.1) to [bend left=35] (-5.25,-3.4);
\draw[very thick] (-4,-6.2) to  (7.7,-0.25);
\draw[red,very thick] (-4,-6.2) to [bend left=35] (7.7,-0.25);
\draw[red,very thick] (-6.1,-4) to [bend left=15]  (3.9,6.25);

%\draw[blue,very thick] (-5.6,-1) to [bend right=15] (0,6.1);
\draw[blue,dashed, very thick] (-5.6,-1) to [bend left=15] (0,6.1);
\draw[very thick, blue] plot [ smooth, tension=.5]   coordinates{(-5.6,-1) (-3.5,1.7)  (-1.5,3)(0,6.1)};

\node[blue] at (0,7) {$\gamma$};

\node at (4,-5.7) {$\Sigma$};

\node[red] at (2,1) {$S$};

\end{tikzpicture}
    \caption{$S$ and \( \Sigma \) are not smoothly isotopic.}
    \label{Fig19}
\end{figure}
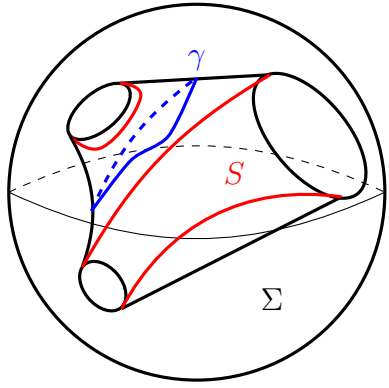

By Definition \ref{def: isotopy minimizer}
there exists a disc $D' \subset U$ with
$\gamma = \partial D' \subset \Sigma \cap B_r(p)$ with the property that $D'$ separates two connected components of $S$.
%In this case, we can find an embedded loop \( \gamma \) in \( \Sigma \) that separates connected components of \( \partial \Sigma \) which are separated by \( S \), as shown in Figure \ref{Fig19}. By Meeks-Yau \cite{MeYa82}, each of these loops bounds an embedded disk \( D \) in \( B_{r_0}(p) \), disjoint from \( S \). 
By Meeks-Yau \cite{MeYa82} we can minimize
in the class of embedded discs that are isotopic to $D'$ in $U \cap B_r(p) \setminus S$ with fixed boundary $\gamma$ to obtain an embedded disc $D$.

By Remark \ref{remalocalpinch}, disc $D$
can be used to deform $\Sigma$ by a neck-pinching family, which is \( L' \)-short, with \( L' \leq L + 2r \). After the deformation, we reduce the number of boundary components of \( \Sigma \) by one. We can now apply the inductive assumption to the new boundary and region we obtained.
%to a number less than \( k+1 \). This concludes the inductive step in the general case.
\end{proof}

\begin{remark}\label{remasmallfol}
By Remark \ref{remalocalpinch}, if in addition to the hypothesis of Lemma \ref{lemma:folsmall} one assumes the surface $\Sigma$ is $L$-short, then for all $\varepsilon$ sufficiently small we can conclude that surfaces $g^{-1}(t)$ will be $L'$-short, for an $L'\leq L+3r$.
\end{remark}

 Given a geometrically prime region $N$, let
 $\Sigma_i$ denote the stable minimal boundary
 components of $N$. For $A>0$, we consider the set of
 stationary integral varifolds of the form:
 $$\mathcal{V}_{N,A}= \left\{V=\sum_j a_{i_j} \Sigma_{i_j},~a_{i_j} \in \mathbb{N},~||V||(N) \leq A \right\}.$$

 Since the coefficients $a_{i_j}\geq0$, the varifold measure of a given $V=\sum_j a_{i_j} \Sigma_{i_j}$ is given by: 
 $$||V||(N)=\sum_j a_{i_j} \Area(\Sigma_{i_j}),$$
and, so, for every $A>0$, the set $\mathcal{V}_{N,A}$ is finite. 

The next Proposition shows that, in a geometrically prime region, using isotopies and neck-pinches in a small ball of fixed radius, one can always decrease the area on an embedded sphere $\Sigma$ by a definite amount unless $\Sigma$, as a varifold, lies in a small neighborhood of $\mathcal{V}_{N,A}$. We will use the $\F$ distance function of varifolds from Pitts \cite{Pitts81}, a metrization of varifold weak topology.

\begin{proposition} \label{surgery}
   Let $N \subset M$ be a geometrically prime region. Given $A>0$ and $\eta_0>0$, there exist $r_0 \in (0, \text{\rm convrad}(M)/2)$ and $\eps_0>0$
   with the following properties.
    Let $\Sigma$ be an embeddded sphere and $U$
    connected component of $N \setminus \Sigma$, such that
    \begin{itemize}
        \item $\partial U$ is mean convex;
        \item $\Area(\Sigma) \leq  A$;
        \item $\F(\Sigma, \mathcal{V}_{N,A}) > \eta_0$.
    \end{itemize}
    Then there exists a ball $B_{r_0}(p)$, such that $\Area(\Sigma \cap B_{r_0}(p))-\mathcal{A}(\Sigma \cap B_{r_0}(p),U \cap B_{r_0}(p))> \varepsilon_0$.
   % Let $S \subset U \cap B_{r_0}(p)$
   % be an isotopy minimizer for $\Sigma \cap B_{r_0}(p) $ in $U \cap B_{r_0}(p)$.
   % Then 
   % $\Area(\Sigma \cap B_r(p))- Area(S)> \eta$.
\end{proposition}

\begin{proof}
We argue by contradiction.
    Suppose the result does not hold. 
    Then there exists a sequence $\Sigma_i\subset U_i$, such that 
    \begin{equation}\label{eq: area deficit}
         \Area(\Sigma_i \cap B_{r_0}(p))-\mathcal{A}(\Sigma \cap B_{r_0}(p),U_i \cap B_{r_0}(p))\leq 2^{-i}
    \end{equation}
    for all $p\in N$.

Since the areas of $\Sigma_i$ are bounded by $A$
there exists a subsequence of $\{ |\Sigma_i| \}$ that converges to a varifold $V$.
    We claim that $V$ is a stationary varifold.
    Indeed, otherwise, there exists a vector field $W$ supported in $B_{r_0}(p)\cap U_i$ and an isotopy $F(\cdot,t) $ of $N$  generated by $W$ (the isotopy  is equal to 
    identity outside $B_{r_0}(p) \cap U_i$)
    that reduces areas of $\Sigma_i$ for all sufficiently large $i$.
   Observe that the condition (\ref{eq: area deficit}) implies that after taking a subsequence we may assume
   that $\{\Sigma_j \}$ is $1/j$-a.m. in the sense of \cite[Definition 3.2]{CD}.
   By \cite[Theorem 7.1]{CD} the support of $V$ is a smooth minimal surface.
\end{proof}

We will also need a lemma that allows us to foliate regions bounded by spheres of very small area by $L$-short spheres.

\begin{lemma} \label{very small area}
    Let $M$ be a Riemannian 3-manifold and $\delta>0$. There exists $\varepsilon(M, \delta)>0$ with the following property.  
    Let $U \subset M$ be a
    subset of $M$ with
    $Vol(U) < Vol(M \setminus U)$ and 
    $\partial U =  \Sigma$
    an $L$-short mean convex 2-sphere with $\Area(\Sigma) < \varepsilon$.

Then there exists a tree-foliation $\{\Sigma_t\}_{t \in T}$ of $U$, starting with $\Sigma_{t_0} = \Sigma$, by $L'$-short surfaces 
for $L'\leq L + \delta$.
\end{lemma}

\begin{proof}
Fix $r\in (0, \frac{\delta}{2})$, so that $r$ is smaller than the convexity radius
of $M$ and every ball of radius $r$
is $2$-bilipschitz diffeomorphic to a Euclidean ball of the same radius.
Suppose $\varepsilon>0$ is such that $\frac{\varepsilon}{r}<\frac{r}{100}$.

Let $K =\lfloor \frac{2\diam(M)}{r} \rfloor$.
Let $p \in \Sigma$. 
%Let $B_M(p,r)= \{x: dist_M(x,p) \leq r\}$, where $dist_M$ denotes the distance in $M$.
Fix a point $p \in \Sigma$.
By coarea inequality we can find balls $\{B_{r_i}(p)\}_{i=1}^K$ with $r_i \in [ir/5, (i+1)r/5]$
and
$$\length(\partial B_{r_i}(p) \cap \Sigma) \leq \frac{r}{20}.$$
Let $\gamma$ denote a connected component of $  \Sigma \cap \partial B_{r_i}(p)$. From the length estimate  we have that $\gamma$ is contained in a ball $B_{r/40}(q)$ for some $q \in \gamma$. By \cite{MeYa82}  there exists an embedded disc $D_{\gamma} \subset U$  that
minimizes area among all discs in $U$
with boundary $\gamma$. 
%Since $D_{\gamma}$ is area minimizing we have $\Area(D_\gamma)\leq \frac{\Area(\Sigma)}{2}< \frac{\varepsilon}{2}$ and by the filling radius bound \cite[Corollary 2.4]{LiMa} the intrinsic diameter $\diam_{D_\gamma}(D_\gamma)$ satisfies $$\diam_{D_\gamma}(D_\gamma)< \frac{\length(\gamma)}{2}+\frac{4\pi}{3} \leq \frac{r}{40}+\frac{4\pi}{3}$$
It follows from the isoperimetric inequality and monotonicity formula for minimal surfaces that, assuming $\varepsilon$ is sufficiently small, we have \begin{equation}\label{D_gamma}
    D_{\gamma} \subset U \cap B_{r/4}(q)
\end{equation}

%we have that $U \cap Busing that any ball of radius $r$ is 2-biLipschitz to an Euclidean ball of same radius, we may minimize area inside a ball of radius $r/2$, and conclude that $D_\gamma$ is also contained in such ball. Since $\gamma$ is a curve in $\partial U=\Sigma$, which is mean convex, we also have that $D_\gamma\subset U$ and it satisfies area bound $\Area(D_\gamma)< \varepsilon$ and $\diam(D_\gamma)<\frac{r}{40}+\frac{2\pi}{3}$.

Consider a connected component $C_k$ of $\Sigma \setminus \bigcup_i \partial B(p,r_i)$ that does not contain $p$. Since $\Sigma$ is a 2-sphere there will be exactly one connected component of $\partial C_k$ (let's call it $\gamma_1$) that is closer to $p$ than all other connected components of $\partial C_k$. By construction, we have that for every point $x \in C_k$ we have 
\begin{equation} \label{eq: dist to gamma1}
    dist_M(x,\gamma_1)<\frac{2r}{5}
\end{equation}

Let $\Sigma_k= C_k \cup \bigcup_j D_{\gamma_j}$,
where $\gamma_j$ denotes the connected components of $\partial C_k$.
%Let $\Sigma_i$ be a connected surface contained
From (\ref{D_gamma}) we have that each $D_{\gamma_j}$ is contained in a ball of radius $\leq \frac{r}{4}$ centered around a point on $\partial D_{\gamma_j}$. Combining this with (\ref{eq: dist to gamma1}) we obtain that
$\Sigma_k \subset B_r(p')$ for some $p' \in \gamma_1$.
Since $r>0$ is less than the convexity radius
of $M$
any surface contained in $B_r(p')$ is $L$-short
for $L< 2r$. Let $V_k$ denote the region bounded by $\Sigma_k$ in $B_r(p')$. By Alexander's theorem there exists a diffeomorphism
$\Phi: B(1) \rightarrow V_k$, where $B(1)$ denotes the unit ball in $\R^3$. Then
$$S_t = \Phi(\{(x,y,x): x^2+y^2 +z^2 = (1-t)^2\})$$
defines a foliation of $V_k$ by $L$-short 2-spheres.

By Remark \ref{remalocalpinch},
there exists a neck-pinching family starting on $\Sigma$ along discs 
$\{D_\gamma\}$ that is $L'$-short for $L'\leq L+\delta$.
%By \cite[Proposition 3.10.]{LiMa} there exists a tree foliation of the region bounded by $\Sigma_i$, starting on $\Sigma_i$. 
\end{proof}

\subsection{Proof of Theorem \ref{thm: L-short foliation}}
By approximation, it is enough to prove the result for generic metrics (in the sense of Baire Category). Thus, by a result of White \cite{White:bumpy}, we may assume the metric $g$ is such that all of its closed embedded minimal surfaces have no nontrivial Jacobi fields. As in the proof of \cite[Theorem 2.7]{LiMa}, we decompose $M$ as a union of geometrically prime regions $\{M_i\}$ with disjoint interiors, and since $M$ is diffeomorphic to the 3-sphere, for each $i$, the large boundary component $S_i$ of $\partial M_i$ is a minimal 2-sphere of Morse index $1$ and all other boundary components are stable minimal 2-spheres. Since $R\geq6$, by classical diameter and area estimates for minimal surfaces (cf., \cite[Theorem 2.1, Corollary 2.4]{LiMa}), the large boundary components $S_i$ have area at most $4\pi$ and diameter at most $\frac{4\pi}{3}$, and all other components have area at most $\frac{2\pi}{3}$ and diameter $\frac{2\pi}{3}$.
 
We will define a tree-foliation of each $M_i$ by $L$-short spheres for an $L<4500$. It is clear that
this implies existence of an $L$-short tree-foliation of $M$.

Now we describe the construction of the tree-foliation for $M_i$. We start by using the first eigenfunction of the 
stability operator, as in the proof of \cite[Lemma 3.2]{MaNe12} to deform the large boundary component $S_i$ to a surface $\Sigma^0$ in the interior
of $M_i$ via an area-decreasing deformation in the tubular neighbourhood 
of $S_i$. Note that $\Sigma^0$ will have mean curvature pointing away from $S_i$ at every point. 
Using Proposition \ref{surgery}, we then define a sequence of area-decreasing deformations and surgeries.
At the end of each step, we obtain a (possibly disconnected) surface  $\Sigma^k= \bigsqcup_j \Sigma^k_j$ of area 
\begin{equation*}
\Area(\Sigma^k) \leq \Area(\Sigma^0) - k\eps_0.
\end{equation*}
with mean curvature pointing away from $S_i$ (we allow for number of connected components of $\Sigma^k$ to possibly increase with $k$).

The construction proceeds by induction on $k$.
Let $l_0 = 14.5$ and $\eta_0>0$ be a small constant depending on $M_i$ to be
specified later. Let $r_0 \in (0, \frac{\text{\rm convrad}(M)}{2})$ and $\varepsilon_0$ 
be constants depending on $\eta_0$ that satisfy
the conclusions of Proposition \ref{surgery}. Without any loss of generality can assume 
$r_0, \varepsilon_0< 0.1$.
We suppose that we have constructed a tree foliation of $U_{k-1} \subset M_i$,
such that the following holds:

\begin{itemize}
    \item[($i$)] $M_i \setminus U_{k-1} = \bigsqcup V_j$, where each $V_j$
    is a geometrically prime region;
    \item[($ii$)] the large boundary 
    component $\Sigma^{k-1}_j$ of $V_j$ satisfies 
    $\diam(\Sigma^{k-1}_j) \leq l_0$ and $\Area(\Sigma^{k-1}_j) \leq  \Area(\Sigma_0) - (k-1)\eps_0$;
    \item[($iii$)]  $\Sigma^{k-1}_j$
is at a distance greater than $\eta_0$ in $\F$ varifold distance
from every stable minimal sphere in $\partial M_i$. 
\end{itemize}

The new surface $\Sigma^k$ will be obtained by applying the following procedure to each connected component  $\Sigma^{k-1}_j$ of $\Sigma^{k-1}$. Since $\diam(\Sigma^{k-1}_j) \leq l_0$ and, by hypothesis ($ii$),
$$\Area(\Sigma^{k-1}_j) \leq  \Area(\Sigma_0) - (k-1)\eps_0<\Area(\Sigma_0)\leq 4\pi < l_0^2,$$
Lemma \ref{lemma: small spheres are L-short} gives that $\Sigma^{k-1}_j$ is $L$-short for $L\leq 2958$. Using hypothesis $(iii)$, we apply, for each $j$,
 Proposition \ref{surgery} to find a ball $B_{r_0}(p)$ and 
surface $S\subset V_j$, so that $S \setminus B_{r_0}(p)=
\Sigma^{k-1}_j \setminus B_{r_0}(p)$
and $S \cap B_{r_0}(p)$ is an isotopy minimizer for
$\Sigma^{k-1}_j \cap B_{r_0}(p)$ in $B_{r_0}(p)$, satisfying
$$\Area(S \cap B_{r_0}(p) )< \Area (\Sigma^{k-1}_j \cap B_{r_0}(p))- \varepsilon_0.$$
By Lemma \ref{lemma:folsmall}, there exists a surface $\hat{\Sigma}_j^{k-1}$ such that
\[
\Area(\hat\Sigma_j^{k-1}) \leq \Area(\Sigma_j^{k-1} \setminus B_r(p)) +\Area(S \cap B_{r_0}(p)) \leq \Area(\Sigma^0) - k\eps_0.
\]

Since each \( {\Sigma}_j^{k-1} \) is \( L \)-short with \( L = 2958 \), it follows from Remark \ref{remasmallfol} that the Morse foliation described in Lemma \ref{lemma:folsmall}, which spans the region between \( \hat{\Sigma}_j^{k-1} \) and \( {\Sigma}_j^{k-1} \), has leaves that are \( L \)-short for \( L = 2958 + 3r_0 \).

Now we come to the key difficulty in the proof: the operation of replacing a part of the surface with an isotopy minimizer inside a small ball (which is described in Proposition \ref{surgery} and Lemma \ref{lemma:folsmall}) may potentially increase the intrinsic diameter of the surface in a way we cannot control. We introduced the notion of $L$-short surfaces to deal specifically with this problem.

 For each connected component $\hat S$ of $\hat{\Sigma}_j^{k-1}$, we now consider two possibilities: if the intrinsic diameter $d$ of $\hat S$ satisfies $d \leq l_0=14.5$, then $\hat S$ bounds a new geometrically prime region, which we label  $\Sigma^k_{j'}$, and it satisfies the hypotheses of the inductive step. If $d>14.5$, we pick points $p_0$ and $q_0$ at a distance $d$ in $\hat S$ and consider boundaries of geodesic balls $\partial B_{p_0}^{\hat S}(r)=\partial \{dist_{\hat S}(x,p)\leq r: x \in \hat S \}$, $r\in (0,d]$. For each integer $l=1,\ldots, \lceil d \rceil$, consider the region $A_l= B_{p_0}^{\hat S}(l)\setminus B_{p_0}^{\hat S}(l-1)$. Since $\Area(\hat{S})<4\pi$, by the coarea formula, for each $l$, there exists a radius $r_l\in (l-1,l)$, such that  $\partial B_{p_0}^{\hat S}(r_l)$ is a finite collection of embedded smooth curves of total length at most $2\pi$. Observe that $\partial B_{p_0}^{\hat S}(r_l)$ may not be connected, and thus we denote its connected components as $\gamma_{l,1}, \gamma_{l,2}, \ldots, \gamma_{l,k_l}$. This way, each connected component of $\hat{S} \setminus \bigcup_{l,i} \gamma_{l,i}$ has area at most $4\pi$ and boundary curves of length at most $2\pi$. For such component $C$, we glue in area-minimizing disks $D_{m}$ along the boundary curves of $C$, and perform neck-pinching isotopies along each $D_{m}$. By the filling radius estimate, each disk $D_{m}$ has diameter at most 
 $$\frac{2 \pi }{3} + \frac{\length(\partial D_{m})}{2}+\frac{2 \pi }{3}\leq 7\pi/3$$
 and since the area of $\hat S$ is at most $4\pi$, we have $\Area(D_{m}) \leq \frac{\Area(\hat S)}{2}\leq 2 \pi$. Thus, by Proposition \ref{prop: neck-pinching L-short}, the surfaces obtained in the process of neck-pinching will all be $L$-short for %$L\leq 2958$. 
$$L\leq 2958+3r_0+2 \pi + 204\frac{7 \pi}{3} + 4 \varepsilon_0<4500$$

After pinching, each connected component $E$ is obtained by a small area-decreasing deformation of $C \cup \bigcup_{m \in N_C }D_{m}$,
where $N_C$ denotes the set of indices with the property
$\partial D_m \subset \partial C$ 
for all $m \in N_C$,
which is a sphere of area at most $\Area(\Sigma_0) - k\eps_0< 4 \pi$ and diameter $d<14.5$. First we explain why the area bounds hold. Let $\tilde{D}_{m} \subset \hat S$ denote a disc in $\hat S$ with $\partial \tilde{D}_{m} = \partial D_{m}$ and such that the interior of 
$C$ is disjoint from the interior of 
$\tilde{D}_{m}$. Since 
$D_{m}$ is area minimizing in the class of
discs isotopic to $\tilde{D}_{m}$ with fixed boundary we have
$$\Area(E) \leq 
\Area(C \cup \bigcup_{m  \in N_C} \tilde{D}_{m})=
\Area(\hat{S})\leq  \Area(\Sigma_0) - k\eps_0< 4 \pi$$
Now we explain the diameter bound.
Let $p_1,p_2 \in E$. Observe that since $S$
is a sphere there exists a unique connected component of
$\partial C$ that is closer to $p$ than all other connected components $\gamma$ of 
$\partial C$. We have $dist_{\hat S}(p_i, \gamma) < \frac{2 \pi}{3} + 2$, and thus: $$dist(p_1,p_2) < 2\left(\frac{2 \pi}{3} + 4\right)  + \frac{1}{2}\textrm{L}(\gamma)<14.5.$$
After doing this on each connected component $\Sigma^{k-1}_j$ of $\Sigma^{k-1}$, and taking the union of new 
$\Sigma^{k}_{j'}$, we obtain the surface $\Sigma^k$.

Since for each $\Sigma^k_j$ we have $\Area(\Sigma^k_j) \leq \Area(\Sigma_0) - k\eps_0
$, it is clear that the inductive process cannot go on forever. It stops when the large components of the geometrically prime regions still to be foliated do not satisfy the hypothesis $(iii)$, that is, its $\F$-distance to the stable components of the boundary must be small. This implies that $\Sigma^k_j$ is $\eta_0$-close in the varifold sense to a union $\G$ of some of the stable minimal boundary components of $M_i$ (with some integer multiplicities). Since varifold convergence does not imply Hausdorff convergence, we argue as follows. Recall, both $\Sigma^k_j$ and each connected component of $\G$ have bounds on area and diameter and thus are $L$-short for $L=2958$. We consider equidistant surfaces $\G_\delta = \partial N_\delta(\G)$, $\delta>0$, to $\G$ and their intersections with $\Sigma_k$. Given $\delta \in (0,r_0)$ small and $\varepsilon_1< \frac{\delta^2}{100}$, we may choose $\eta_0=\eta_0(\delta, \varepsilon_1)$ such that $\Sigma^k_j$ being $\eta_0$-close to $\G$ in varifold sense implies that  $\Sigma^k_j \setminus N_\delta(\G)$ has area less than
$\varepsilon_1$. 
Using the coarea inequality we can find
$h \in (\delta, 2 \delta)$, such that
the intersection $\Sigma^k_j \cap \partial N_h(\G)$
has length less than $\frac{\delta}{10}$. By the isoperimetric inequality each connected component of
$\Sigma^k_j \cap \partial N_h(\G)$
bounds a small area disc inside 
$\partial N_h(\G)$.
Starting with the innermost boundary curve we can apply
 Remark \ref{remalocalpinch} to define neck-pinching families along these small discs. In the end, we obtain a tree foliation that starts on $\Sigma^k_j$
 and deforms it to a disjoint union
 of surfaces
 $\Sigma^L \cup \Sigma^S$, where $\Sigma^L \subset N_{2 \delta}(\G)$
 and $\Sigma^S$ satisfies
 $\Area(\Sigma^S)< \varepsilon_1$.
 
We apply Lemma \ref{very small area}
to foliated small regions bounded by connected components of $\Sigma^S$. Finally, we can foliate regions between $\Sigma^L$ and $\Gamma$ in an arbitrary way. All surfaces in the foliation will be $L+ 2 \delta$-short for $L<4500$ since they are all contained in the $\delta$-tubular neighborhood of a connected component of $\G$, which is $L$-short.
%Thus, we may finish foliating $M_i$ by foliating these small regions and $S_\delta$ which will be $L$-short for $L<3000$. This finishes the proof of Theorem \ref{thm: L-short foliation}.

\section{Parametric sweepout of $L$-short spheres} \label{sec: parametric}

Let $M$ be a Riemannian $3$-sphere with scalar curvature bounded below by  $6$.
Let $\Omega M$ denote the space of closed, piecewise differentiable curves on $M$, and $\Omega_xM$, the subspace of closed curves of length at most $x$. (Note that $\Omega_0M$ will then denote the space of constant curves, which can be identified with $M$ itself.)

In this section we will construct a sweepout of $M$ by ``short" curves, $i.e.$, a map $f:D^2 \longrightarrow \Omega_x M$ where $\partial D^2$ is mapped to $\Omega_0M$ and that is not contractible over $2$-disks with fixed boundary. Here $x = 10 ^5$, which will immediately imply the proof of Theorem ~\ref{thm:mainC}.

\begin{theorem} \label{sweepout of M}
    Suppose $M = (S^3,g)$ is a Riemannien 3-sphere
    with scalar curvature $\geq 6$.
    There exists a noncontractible family 
    $$f: (D^2, \partial D^2) \rightarrow (\Omega_x M, \Omega_0 M),$$ where $x = \finalbound$.
\end{theorem}

%One of the building blocks in the construction of this sweepout is the following Theorem (Theorem 1.2 in 
%~\cite{LNR15}):

%\begin{theorem} \label{TheoremYNR}
%Let $(D^2,g)$ be a Riemannian $2$-disk. Let $p \in \partial D^2$. Then there exists a homotopy $\gamma_t$ of loops based at $p$ between $\gamma_0=\partial D^2$ and $\gamma_1=p$, such that 
%$|\gamma_t| \leq 2 |\partial D^2|+686\sqrt{\Area(D^2)}+2\diam(D^2)$. 

%Here $|\gamma_t|$ and $|\partial D^2|$ denote lengths of $\gamma_t$ and $\partial D^2$ respectively, and 
%$\diam(D^2)$ denotes the diameter of $D^2$.
%\end{theorem}

To prove Theorem \ref{sweepout of M}.
 we will fist define a family $F: (K, \partial K) \rightarrow (\Omega M,  \Omega_0 M)$, where $K$ is a 2-complex homeomorphic to a disc $D^2$, of potentially very long closed curves, but with the property that each curve $F(x)$ lies on an $L$-short $2$-sphere. The construction of $F$ proceeds as follows. By Theorem ~\ref{thm: L-short foliation} there exists a tree-foliation parametrized by a tree $T$,  $\{ \Sigma_x\}_{x \in T}$, and each 
$\Sigma_x$ is $L$-short with $L<4500$.

\begin{figure}[h!] 
    \centering
 

\begin{tikzpicture} [scale=0.5] 

\draw[thick]  (7,.5) -- (7,10.5);

\draw[->,thick] (1,4) to [bend left=25] (6,4);
\node at (3.5,5.5) {$\large S$};

\node[draw,circle,inner sep=2pt,fill=black] at (-3,10) {};
\node[draw,circle,inner sep=2pt,fill=black] at (-3,8){};

\draw[]   (-3,8)--(-3,10);

\node[draw,circle,inner sep=2pt,fill=black] at (-1,6){};
\node[draw,circle,inner sep=2pt,fill=black] at (-6,5){};

\draw[]   (-3,8)--(-1,6);
\draw[]   (-3,8)--(-6,5);

\node[draw,circle,inner sep=2pt,fill=black] at (-5,2.5){};
\node[draw,circle,inner sep=2pt,fill=black] at (-7,3.5){};

\draw[]    (-6,5)--(-5,2.5);
\draw[]    (-6,5)--(-7,3.5);

\node[draw,circle,inner sep=2pt,fill=black] at (-2,5){};
\node[draw,circle,inner sep=2pt,fill=black] at (0,4){};

\draw[]    (-1,6) -- (-2,5);
\draw[]    (-1,6) -- (0,4);

\node[draw,circle,inner sep=2pt,fill=black] at (-2,2){};
\node[draw,circle,inner sep=2pt,fill=black] at (1,3){};

\draw[]    (0,4) -- (-2,2) ;
\draw[]    (0,4) -- (1,3) ;

\node[draw,circle,inner sep=2pt,fill=black] at (0,1){};
\node[draw,circle,inner sep=2pt,fill=black] at (2,1){};

\draw[]    (2,1) -- (1,3) ;
\draw[]    (0,1) -- (1,3) ;

 \end{tikzpicture}
 
\caption{$S:T \longrightarrow \R$}
\label{Fig.3}
\end{figure}

%Define inductively a height function on  $T$, $S: T \longrightarrow \R $, (see Fig. ~\ref{Fig.3}), starting from a leaf and making sure that each two subsequent edges lie below the prior ones. 
Recall that a tree foliation 
$\{\Sigma_x \}_{x \in T}$ comes with maps $f:M \to T$  and $S: T \longrightarrow \R $, such that $\Si_x=f^{-1}(x)$ and $S \circ f$ is a Morse function.

\begin{figure}[h!] 
    \centering
\begin{tikzpicture}[scale=.75]

%1stside
\draw (-1,6) to (1,6);
\draw (-1,3) to (1,3);
\draw (-1,6) to (-1,3);
\draw (1,6) to (1,3);

\draw[dashed] (-1,3) to (-3,2);
\draw[dashed] (-0,3) to (-1,2);
\draw[dashed] (1,3) to (3,2);
\draw[dashed] (0,3) to (1,2);

\draw(-3,2) to (-1,2);
\draw(-3,-1) to (-1,-1);
\draw (-3,-1) to (-3,2);
\draw (-1,-1) to (-1,2);

\draw(3,2) to (1,2);
\draw(3,-1) to (1,-1);
\draw (3,-1) to (3,2);
\draw (1,-1) to (1,2);

\draw[dashed] (-3,-1) to (-4,-2);
\draw[dashed] (-2,-1) to (-2.25,-2);
\draw[dashed] (-2,-1) to (-2,-2);
\draw[dashed] (-1,-1) to (-0.25,-2);

\draw[dashed] (3,-1) to (4,-2);
\draw[dashed] (2,-1) to (2.25,-2);
\draw[dashed] (2,-1) to (2,-2);
\draw[dashed] (1,-1) to (0.25,-2);

\draw(-4,-2) to (-2.25,-2);
\draw(-4,-5) to (-2.25,-5);
\draw(-4,-2) to (-4,-5);
\draw(-2.25,-2) to (-2.25,-5);

\draw(-2,-2) to (-0.25,-2);
\draw(-2,-5) to (-0.25,-5);
\draw(-2,-2) to (-2,-5);
\draw(-0.25,-2) to (-0.25,-5);

\draw(2,-2) to (0.25,-2);
\draw(2,-5) to (0.25,-5);
\draw(2,-2) to (2,-5);
\draw(0.25,-2) to (0.25,-5);

\draw(4,-2) to (2.25,-2);
\draw(4,-5) to (2.25,-5);
\draw(4,-2) to (4,-5);
\draw(2.25,-2) to (2.25,-5);

%2ndside

\draw (7,5) to (9,5);
\draw (7,2) to (9,2);
\draw (7,2) to (7,5);
\draw (9,2) to (9,5);

\draw (7.5,2) to (7.5,5);
\draw (8,2) to (8,5);
\draw (8.5,2) to (8.5,5);
\draw (7,4) to (9,4);
\draw (7,3) to (9,3);

\draw (7,2) to (6,0);
\draw (8,2) to (7,0);
\draw (6,0) to (7,0);
\draw (6.5,0) to (7.5,2);
\draw (6.5,1) to (7.5,1);

\draw (9,2) to (10,0);
\draw (8,2) to (9,0);
\draw (10,0) to (9,0);
\draw (9.5,0) to (8.5,2);
\draw (8.5,1) to (9.5,1);

\draw (6,0) to (5.5,-3);
\draw (6,-3) to (6.5,0);
\draw (5.5,-3) to (6,-3);
\draw (6.5,0) to (7,-3);
\draw (7,-3) to (7.5,-3);
\draw (7.5,-3) to (7,0);

\draw (5.67,-2) to (6.16,-2);
\draw (5.834,-1) to (6.334,-1);

\draw (6.833,-2) to (7.333,-2);
\draw (6.67,-1) to (7.17,-1);

\draw (9,0) to (8.5,-3);
\draw (9,-3) to (9.5,0);
\draw (8.5,-3) to (9,-3);
\draw (9.5,0) to (10,-3);
\draw (10,-3) to (10.5,-3);
\draw (10.5,-3) to (10,0);

\draw (8.67,-2) to (9.16,-2);
\draw (8.834,-1) to (9.334,-1);
\draw (9.833,-2) to (10.333,-2);
\draw (9.67,-1) to (10.17,-1);

%labels

\draw[->] (4,1.5) to [bend left=15] (6,1.5);
\node at (6.5,3) {$K$};

\node at (2,4.5) {$e_1\times I$};
\node at (-4,.5) {$e_2\times I$};
\node at (4,.5) {$e_3\times I$};
\node at (-4,-5.7) {$e_4\times I$};
\node at (-1.2,-5.7) {$e_5\times I$};
\node at (1.2,-5.7) {$e_6\times I$};
\node at (3.2,-5.7) {$e_7\times I$};

\end{tikzpicture}
\caption{Constructing $K$}
\label{Fig.4}
\end{figure}
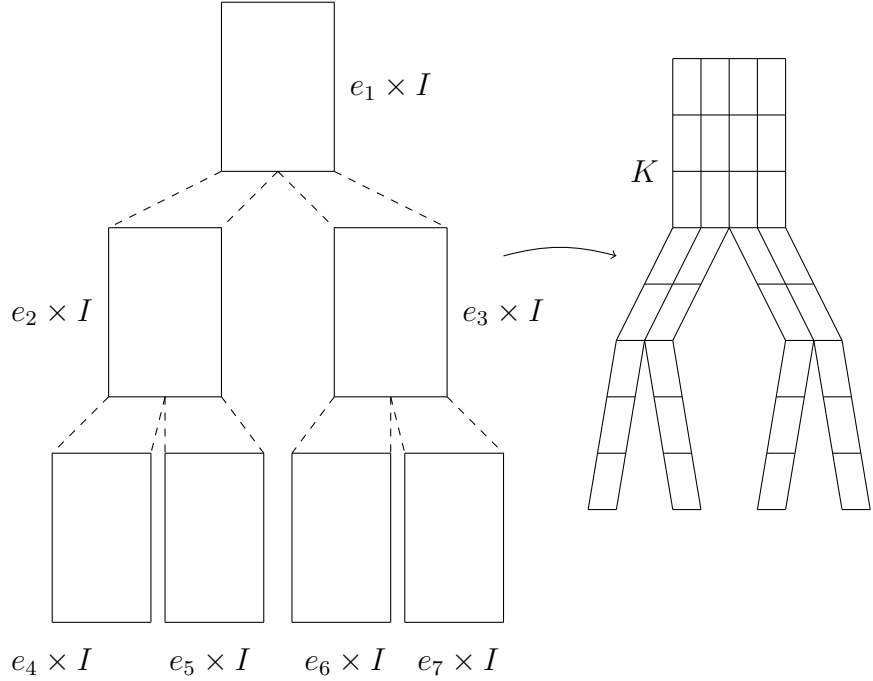

We will next construct a CW complex $K$, corresponding to $T$ in the following way. First for each edge $e_i$ of $T$ we will construct a rectangle $R_i=e_i \times [0,1]$. 
%The function $S$ defined above induces the top and the bottom edges of each $R_i$. 
We will now glue the rectangles that correspond to the neighboring edges of $T$ in the following way. Consider an edge $e_i$ and the two edges $e_j$ and $e_k$ situated directly below $e_i$. We will identify the lower edge of $R_i$ with the upper edges of $R_j$ and 
$R_k$ via positive linear maps $p_j: [0,1] \longrightarrow [0, \frac{1}{2}]$, $p_2: [0,1] \longrightarrow [\frac{1}{2},1]$. The resulting complex will be a disk, (see Fig. ~\ref{Fig.4}). 

% will need a figure here. 

%\noindent 1. $G^{-1}(x)$ is a sphere of dimension $2$ of area at most... and intrinsic diameter at most...., if $x \in e$, where $e$ is an edge of $T$

%\noindent 2. For any vertex $v \in T$, $G^{-1}(v)$ is a wedge of $2$-spheres, each of which is satisfying the condition above. 

%\noindent 3. If $w \in T$ is a leaf, then $G^{-1}(w)$ is a point. 

%\noindent 4. $G$ generates a sweepout of $M$ by $2$-spheres. 

Let us now describe a possibly long sweepout $F: K \longrightarrow \Omega M$ by closed curves. 
The boundary of $K$ will be mapped to constant curves, while for a given rectangle $e_i \times I$ in $K$ and each horizontal segment $t \times I$, family $F(x), x \in (t,I)$, will be  generated by a sweepout of a sphere $\Sigma_t$, which will vary continuously with the sphere. 

We describe how this family is constructed in the neighbourhood of the bottom side of a rectangle $e_i \times I$, where it is glued 
to the top sides of rectangles $e_j \times I$
and $e_k \times I$. 
Let $v$ denote the vertex
of $T$ where edges $e_i,e_j,e_k$ meet. 
Relabeling if necessary we may assume that 
$\{\Sigma \}_{t\in e_i}$ undergoes a neck-pinch singularity at $v$.
Let $p$
denote the singular point of $\Sigma_v$.
By the definition of tree foliation
for $t \in e_i$ with $t \to v$ there exists
a family of simple closed curves $\gamma_t \subset \Sigma_t$ converging
to $p$ with $L(\gamma_t) \rightarrow 0$.
Let $D^1_t$ and $D^2_t$ denote the two connected components of $\Sigma_t \setminus \gamma_t$.
We fix two continuous families of diffeomorphisms 
$\Phi_t^i:D(1) \rightarrow  D^i_t$, $i=1,2$, where $D(1)$
denotes the unit disc in $\R^2$,
with $\Phi^1_t|_{\partial D(1)} = \Phi^2_t|_{\partial D(1)}$. We
define $F(t,s) = \Phi_t^1(\{x^2+y^2=(2s)^2\})$
for $s \leq \frac{1}{2}$ and
$F(t,s) = \Phi_t^2(\{x^2+y^2=(2-2s)^2\})$
for $s > \frac{1}{2}$. Note that $F(t, \frac{1}{2}) = \gamma_t \to p$  as $e_i \ni t \to v$.
For $t \in e_j$ very close to $v$ we can continuously extend
family $\{ F(v,s)\}_{s \in [0,\frac{1}{2}]}$
to a sweepout of $\Sigma_t$; and for 
$t \in e_k$ very close to $v$ we can continuously extend
family $\{ F(v,s)\}_{s \in [\frac{1}{2},1]}$
to a sweepout of $\Sigma_t$.
The construction is illustrated in Fig. ~\ref{Fig.8}.

\begin{figure}[h!] 
    \centering
\begin{tikzpicture}[scale=.5]
\draw[name path=sphere 1]  plot [smooth cycle,tension=1] coordinates {(-10,0) 
    (0,8) (10,0) (0,-9)};    
\path[name path=line 1] (-7,-9) -- (-7,0);
\path [name intersections={of=sphere 1 and line 1,by=A}];
\path[name path=line 2] (-7,9) -- (-7,0);
\path [name intersections={of=sphere 1 and line 1,by=A}];
\path [name intersections={of=sphere 1 and line 2,by=B}];
\draw[red] (A) to [bend right=30] (B);
\draw[densely dotted, red] (A) to [bend left=30] (B);

\path[name path=line 3] (-3.5,-9) -- (-3.5,0);
\path[name path=line 4] (-3.5,9) -- (-3.5,0);
\path [name intersections={of=sphere 1 and line 3,by=C}];
\path [name intersections={of=sphere 1 and line 4,by=D}];
\draw[red] (C) to [bend right=30] (D);
\draw[densely dotted, red] (C) to [bend left=30] (D);

\path[name path=line 5] (-0,-9) -- (-0,0);
\path[name path=line 6] (-0,9) -- (-0,0);
\path [name intersections={of=sphere 1 and line 5,by=E}];
\path [name intersections={of=sphere 1 and line 6,by=F}];
\draw[red] (E) to [bend right=20] (F);
\draw[densely dotted, red] (E) to [bend left=20] (F);

\path[name path=line 7] (3,-9) -- (3,0);
\path[name path=line 8] (3,9) -- (3,0);
\path [name intersections={of=sphere 1 and line 7,by=G}];
\path [name intersections={of=sphere 1 and line 8,by=H}];
\draw[red] (G) to [bend right=30] (H);
\draw[densely dotted, red] (G) to [bend left=30] (H);

\path[name path=line 9] (7,-9) -- (7,0);
\path[name path=line 10] (7,9) -- (7,0);
\path [name intersections={of=sphere 1 and line 9,by=I}];
\path [name intersections={of=sphere 1 and line 10,by=J}];
\draw[red] (I) to [bend right=30] (J);
\draw[densely dotted, red] (I) to [bend left=30] (J);

%sphere2
    \draw[name path=sphere 2]  plot [smooth cycle,tension=0.8] coordinates {(-8,0) 
    (-5,5) (.25,2) (5,5) (8,0) (5,-5) (0,-2) (-5,-5) }; 

    \path[name path=line 11] (-4,-6) -- (-4,0);
    \path[name path=line 12] (-4,6) -- (-4,0);  
    \path [name intersections={of=sphere 2 and line 11,by=K}];
    \path [name intersections={of=sphere 2 and line 12,by=L}];
    \draw[red] (K) to [bend right=30] (L);
\draw[densely dotted, red] (K) to [bend left=20] (L);

\path[name path=line 13] (4,-6) -- (4,0);
    \path[name path=line 14] (4,6) -- (4,0);  
    \path [name intersections={of=sphere 2 and line 13,by=M}];
    \path [name intersections={of=sphere 2 and line 14,by=N}];
    \draw[red] (M) to [bend right=20] (N);
\draw[densely dotted, red] (M) to [bend left=20] (N);

%sphere3&4

\draw[name path=sphere 3]  plot [smooth cycle,tension=0.8] coordinates {(-7,0) 
    (-5,3) (-1.5,.8)  (0,0) (-1.5,-.8) (-5,-3)};

\draw[name path=sphere 4]  plot [smooth cycle,tension=.8] coordinates {(7,0) 
    (5,3) (1.5,.8)  (0,0) (1.5,-.8) (5,-3)};

\path [name intersections={of=sphere 3 and line 11,by=O}];
    \path [name intersections={of=sphere 3 and line 12,by=P}];
    \draw[red] (O) to [bend right=30] (P);
\draw[densely dotted, red] (O) to [bend left=20] (P);

\path[name path=line 13] (4,-6) -- (4,0);
    \path[name path=line 14] (4,6) -- (4,0);  
    \path [name intersections={of=sphere 4 and line 13,by=Q}];
    \path [name intersections={of=sphere 4 and line 14,by=R}];
    \draw[red] (Q) to [bend right=20] (R);
\draw[densely dotted, red] (Q) to [bend left=20] (R);

%sphere6and7

\draw[name path=sphere 5] (-4.3,0) circle[radius=55pt];
\draw[name path=sphere 6] (4.3,0) circle[radius=55pt];

\path [name intersections={of=sphere 5 and line 11,by=S}];
    \path [name intersections={of=sphere 5 and line 12,by=T}];
    \draw[red] (S) to [bend right=30] (T);
\draw[densely dotted, red] (S) to [bend left=20] (T);

\path [name intersections={of=sphere 6 and line 13,by=U}];
    \path [name intersections={of=sphere 6 and line 14,by=V}];
    \draw[red] (U) to [bend right=10] (V);
\draw[densely dotted, red] (U) to [bend left=10] (V);

%nodes

\filldraw[red] (-10,0) circle[radius=3pt];
\filldraw[red] (-8,0)
circle[radius=3pt];
\filldraw[red] (-7,0)
circle[radius=3pt];
\filldraw[red] (-6.25,0)
circle[radius=3pt];
\filldraw[red] (-2.35,0)
circle[radius=3pt];
\filldraw[red] (2.35,0)
circle[radius=3pt];
\filldraw[red] (6.25,0)
circle[radius=3pt];
\filldraw[red] (10,0) 
circle[radius=3pt];
\filldraw[red] (8,0)
circle[radius=3pt];
\filldraw[red] (7,0)
circle[radius=3pt];
\filldraw[red] (0,0)
circle[radius=3pt];

    \end{tikzpicture}
\caption{Sweepout of a sphere}
\label{Fig.8}
%ΩΩ
\end{figure}
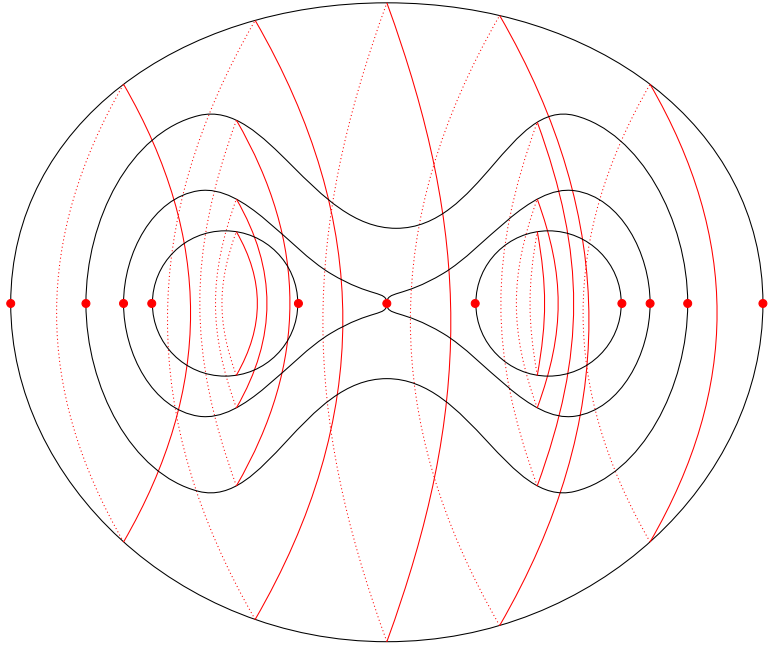

Once we defined $F(x)$ for all $x \in U_v \times I$, where $U_v$ denotes a neighbourhood of a vertex $v \in T$, we can extend to the rest of $K$, since for each edge $e_i$ this amounts to interpolating between two homotopic families in $(\Om M, \Om_0 M)$.

%
%  Sweepout the round sphere by closed curves, satisfying the above conditions. The diffeomorphism between $S^3$ and $M$ will also induce the required sweepout of $M$.  Note that each curve in the sweepout of $M$ will fully lie on one of the spheres $\Sigma_x$, (see Fig. ~\ref{Fig.8}).

This finishes the construction of homotopically non-trivial map $F: (K, \partial K) \rightarrow ( \Omega M,  \Omega_0 M)$, where each curve $F(x)$ lies on an $L$-short sphere.

We will now show that this family can be homotopically 
deformed to a family of short closed curves.
Let $\Pi_x M$ denote the space piecewise smooth maps
from $[0,1]$ to $M$ of length $\leq x$.
Here and below $R(\delta)$ will denote any function of $\delta$ that tends to $0$ as $\delta$ tends to $0$. 

\begin{proposition} \label{mainprop: short arcs}
    Suppose $F: (K, \partial K) \rightarrow ( \Omega M,  \Omega_0 M)$ is a family of closed curves, such that
    each $F(x)$ lies on an $L$-short sphere.
    For every $\delta>0$ there exists a map $H: (K \times [0,1], \partial K  \times [0,1]) 
    \rightarrow (\Pi_{5L +R(\delta)} M, \Pi_0M)$ with the following properties:
\begin{enumerate}
    \item $H(x,0)$ is a constant map with image $F(x)(0)$;
    \item  %there exists a family of non-decreasing piecewise smooth surjective maps $f_x:[0,1] \rightarrow [0,1]$, such that
    $H(x,t)(0) = F(x)(0)$ and $H(x,t)(1) = F(x)(t)$.
\end{enumerate}
\end{proposition}

In other words, for fixed $x$, $H(x,t)$ is a 1-parameter family of arcs connecting $F(x)(0)$ to all other points of $F(x)$. 
$\{ H(x,1) \}_{x \in K}$ is then a family 
of short closed curves.

\begin{proof}[Proof of Theorems \ref{sweepout of M}  and Theorem \ref{thm:mainC} from
Proposition \ref{mainprop: short arcs}]
    Let $F:(K, \partial K) \rightarrow ( \Omega M,  \Omega_0 M)$ be the non-contractible family of closed
    curves constructed above and let $H(x,t)$ be the corresponding family 
    of arcs from Proposition \ref{sweepout of M}.
   Since $H(x,1)(0) = H(x,1)(1)$, identifying the endpoints of $[0,1]$ we obtain a map $\tilde{F}:(K, \partial K) \rightarrow ( \Omega M,  \Omega_0 M)$,
   where $\tilde{F}(x)$ has the same image and length as $H(x,1)$. In particular, the length of $\tilde{F}(x) $ is bounded by $5L + R(\delta) \leq \finalbound$ for all $x \in K$. 

   We have that $F$ is homotopic to $\tilde{F}$
   with the homotopy given by the family of closed curves
   \begin{equation*}
       h(x,t) = H(x,t) * F(x)|_{[t,1]}
   \end{equation*}
   Here $*$ denotes concatenation of two paths.
   This proves Theorem \ref{sweepout of M}.

   It follows from the Morse theory on the space of closed curves \cite{bott1982} that there exists a closed geodesic $\gamma$ in $M$ of length 
   less than $\finalbound + R(\delta)$. Choosing $\delta$ sufficiently small we obtain a bound of $\finalbound$.
\end{proof}

In the rest of this section we prove Proposition \ref{mainprop: short arcs}.

%The existence of this short  sweepout $f:D^2 \longrightarrow \Omega_x M$  will follow from Theorem ~\ref{thm: L-short foliation}, and from a procedure that allows one to construct sweepouts by short curves, provided that given a curve, its points can be continuously  connected to a fixed point with "short" curves. 

The proof will follow from a 
 construction that first appeared in ~\cite{NR13}, where it was done in the case of based point loops or paths spaces on $M$. Here we modify this construction for the free loop space. The result will follow from the sequence of lemmas below. 

\begin{lemma} \label{pathhomotopy}
Let us consider a digon formed by paths $e_1, e_2$  connecting points $p_1, p_2$ of lengths $l_1, l_2$ respectively, (see fig. ~\ref{Fig.13} (a)). Suppose loop $\alpha=e_1*\bar{e}_2$ based at $p_1$, (see fig. ~\ref{Fig.13} (b)) can be contracted over loops $\alpha_\tau$ based at $p_1$ of length at most $l_3$, (see fig. ~\ref{Fig.13} (c)). Then $e_1$
is path-homotopic to $e_2$ over curves of length at most $\min \{l_1, l_2\}+ l_3$.

Moreover, a parametric version of this Lemma holds: 

Suppose we are given a manifold $X$ and families of paths
$\{e_1^x\}_{x \in X}$, $\{e_2^x\}_{x \in X}$  connecting points $p_1(x), p_2(x)$ of lengths $l_1(x), l_2(x)$ respectively, and suppose there exists a family of contractions of loops $\alpha(x)=e_1^x*\bar{e}_2^x$ based at $p_1(x)$,  over loops based at $p_1(x)$ of length at most $l_3(x)$. Then there exists a family of path homotopies of $e_1^x$
to $e_2^x$ over curves of length at most 
$\min_{i=1,2} \max_{x \in X } (l_i(x)+ l_3(x))$.

\end{lemma}

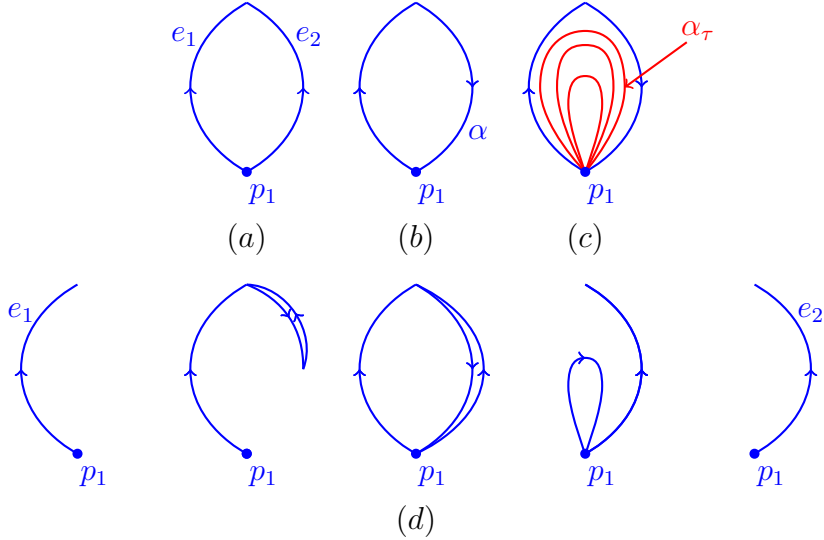
\begin{figure}[h!] 
    \centering 
\begin{tikzpicture}[scale=0.75]
    \begin{scope}[thick,decoration={
    markings,
    mark=at position 0.5 with {\arrow{>}}}
    ] 
    \draw[postaction={decorate},blue] plot [smooth,tension=1.2] coordinates {(-3,-1.5) (-4,0) (-3,1.5)};
    \draw[postaction={decorate},blue] plot [smooth,tension=1.2] coordinates {(-3,-1.5) (-2,0) (-3,1.5)};

    \begin{scope}[thick,decoration={
    markings,
    mark=at position 0.5 with {\arrow{>}}}
    ] 
    \draw[postaction={decorate},blue] plot [smooth,tension=1.2] coordinates {(0,-1.5) (-1,0) (0,1.5)};
    \draw[postaction={decorate},blue] plot [smooth,tension=1.2] coordinates {(0,1.5) (1,0) (0,-1.5)};
    \end{scope}

    \begin{scope}[thick,decoration={
    markings,
    mark=at position 0.5 with {\arrow{>}}}
    ] 
    \draw[postaction={decorate},blue] plot [smooth,tension=1.2] coordinates {(3,-1.5) (2,0) (3,1.5)};
    \draw[postaction={decorate},blue] plot [smooth,tension=1.2] coordinates {(3,1.5) (4,0) (3,-1.5)};

    \draw[red] plot [smooth,tension=1] coordinates {(3,-1.5) (3.7,0) (3,1) (2.2,0) (3,-1.5)}; 
    \draw[red] plot [smooth,tension=1] coordinates {(3,-1.5) (3.5,0) (3,.75) (2.5,0) (3,-1.5)};
    \draw[red] plot [smooth,tension=0.9] coordinates {(3,-1.5) (3.3,-.3) (3,.2) (2.7,-.3) (3,-1.5)};
    \end{scope}

    \filldraw[blue] (3,-1.5) circle[radius=2pt];
    \filldraw[blue] (0,-1.5) circle[radius=2pt];
    \filldraw[blue] (-3,-1.5) circle[radius=2pt];

    \node[blue] at (3.3,-1.9) {$p_1$};
    \node[blue] at (0.3, -1.9) {$p_1$};
    \node[blue] at (-2.7,-1.9) {$p_1$}; 

    \node[blue] at (-4.1,0.9) {$e_1$}; 
    \node[blue] at (-1.9,0.9) {$e_2$}; 

    %\node[blue] at (-1.1,0.9) {$e_1$}; 
    \node[blue] at (1.1,-0.8) {$\alpha$}; 

    \node[red] at (5,1) {$\alpha_\tau$};
    \draw[red,->] (4.8,0.8) to (3.7,0);

    \node at (3,-2.7) {$(c)$};
    \node at (0, -2.7) {$(b)$};
    \node at (-3,-2.7) {$(a)$};

    \begin{scope}[thick,decoration={
    markings,
    mark=at position 0.5 with {\arrow{>}}}
    ] 
    \draw[postaction={decorate},blue] plot [smooth,tension=1.2] coordinates {(-3,-6.5) (-4,-5) (-3,-3.5)};
    \draw[postaction={decorate},blue] (-3,-3.5) to [bend left=35] (-2,-5);
    \draw[postaction={decorate},blue] (-2,-5) to [bend right=55] (-3,-3.5);

    \filldraw[blue] (-3,-6.5) circle[radius=2pt];
    %\node at (-2.7,-6.7) {$p_1$};

    \draw[postaction={decorate},blue] plot [smooth,tension=1.2] coordinates {(0,-6.5) (-1,-5) (0,-3.5)};
    \draw[postaction={decorate},blue] plot [smooth,tension=1.2] coordinates { (0,-3.5) (1,-5) (0,-6.5)};
    \draw[postaction={decorate},blue] plot [smooth,tension=1.2] coordinates { (0,-6.5)  (1.2,-5) (0,-3.5) };
    
    \filldraw[blue] (-0,-6.5) circle[radius=2pt];
    \node[blue] at (0.3,-6.9) {$p_1$};

    \filldraw[blue] (-3,-6.5) circle[radius=2pt];
    \node[blue] at (-2.7,-6.9) {$p_1$};

     \draw[postaction={decorate},blue] plot [smooth,tension=1.2] coordinates {(-6,-6.5) (-7,-5) (-6,-3.5)};

   \filldraw[blue] (-6,-6.5) circle[radius=2pt];
    \node[blue] at (-5.7,-6.9) {$p_1$};

    \draw[postaction={decorate},blue] plot [smooth,tension=1.2] coordinates {(3,-6.5) (4,-5) (3,-3.5)};

   \filldraw[blue] (3,-6.5) circle[radius=2pt];
    \node[blue] at (3.3,-6.9) {$p_1$};

    \draw[postaction={decorate},blue] plot [smooth,tension=1.2] coordinates {(3,-6.5) (4,-5) (3,-3.5)};

     \draw[postaction={decorate},blue] plot [smooth,tension=0.9] coordinates {(3,-6.5) (2.7,-5.3) (3,-4.8)  (3.3,-5.3) (3,-6.5)};

    \draw[postaction={decorate},blue] plot [smooth,tension=1.2] coordinates {(6,-6.5) (7,-5) (6,-3.5)};

    \filldraw[blue] (6,-6.5) circle[radius=2pt];
    \node[blue] at (6.3,-6.9) {$p_1$};
    \node at (0, -7.7) {$(d)$};

    \node[blue] at (-7,-4) {$e_1$}; 
    \node[blue] at (7,-4) {$e_2$};

    \end{scope}   
    \end{scope}
    \end{tikzpicture}

    %\caption{Path homotopy (Lemma 3.2)}
\caption{Constructing path homotopy}
\label{Fig.13}

\end{figure}

\begin{proof}
 Assume we are in the non-parametric case first. Consider the following path-homotopy: $$e_1 \sim e_1 * \bar{e}_2 *e_2 \sim \alpha_t *e_2 \sim e_2$$
 (see fig. ~\ref{Fig.13} (d)).
 Note that the length of curves in this homotopy is at most $l_2+l_3$. Similarly, one can construct a homotopy of length at most $l_1+l_3$. 
 
 Applying the same construction to parametric families we obtain the desired families of path homotopies.
\end{proof}

\begin{lemma} \label{shortcurves}
Consider a quadrilateral formed by two pairs of opposite sides: 
$e_i=e_i(t)$ and $\sigma_i=\sigma_i(s)$, $i=0,1$, (see fig. ~\ref{Fig.5} (a)). Let $l_i = length (e_i)$ and 
$m_i = length (\sigma_i)$. Consider a loop $\alpha$ based at $\sigma_0(\frac{1}{2})=p$ corresponding to the concatenation $\bar{\sigma}_0|_{[\frac{1}{2}, 1]} * e_0 * \sigma_1 * \bar{e}_1 * \bar{\sigma}_0|_{[0,\frac{1}{2}]}$. (Here and below $\bar{\beta}$ will denote the curve $\beta$ travelled in the opposite direction.)
Suppose $\alpha$ is contractible to $p$ along loops $\alpha_\tau$ based at $p$ of length at most $l_3$. Then there exists a one-parameter family of curves 
$\gamma_s$ connecting $\sigma_0(s)$ with a corresponding $\sigma_1(s)$ of length at most $\min\{l_1, l_2\}+l_3+2(m_1 + m_2)$, such that $\gamma_0=e_0$ and $\gamma_1=e_1$.

Moreover, the corresponding parametric version of this Lemma holds.
\end{lemma}

\begin{proof}
    First note that $b_0=\bar{\sigma}_0|_{[\frac{1}{2},1]} * e_0 * \sigma_1|_{[0,\frac{1}{2}]}$ is path-homotopic to 
    $b_1=\sigma_0|_{[\frac{1}{2},1]} * e_1 * \bar{\sigma}_1|_{[0,\frac{1}{2}]}$ by Lemma ~\ref{pathhomotopy} over the curves of length at most $\min\{l_1, l_2\}+l_3 + 2(m_1+m_2)$. One can, for example, consider  the following path-homotopy: $b_0 \sim b_0 * \bar{b}_1 * b_1 \sim \alpha_\tau * b_1 \sim b_1$. 
   % Moreover, the lengths of curves in this homotopy is bounded by $l_2+l_3+m_1+m_2$ (see Lemma ~\ref{shortcurves}). 
   Parametrize the curves in this path-homotopy by $x$ to obtain a $1$-parameter family of curves $b_x$. 
    
    Now, for each $s \in [0,1]$ let us connect the ``opposite" points, $\sigma_0(s)$ with $\sigma_1(s)$ as follows. 
For $s \in [0,\frac{1}{2}]$ One can connect $\sigma_0(s)$ with $\sigma_1(s)$ by the curves 
$\bar{\sigma}_0|_{[s,0]} * e_0 *\sigma_1|_{[0,s]}$, while for $s \in [\frac{1}{2},1]$ (see fig. ~\ref{Fig.5} (b)-(d)). We can connect
the opposite points by curves $\sigma_0|_{[s,1]}*e_1*\bar{\sigma}_1|_{[1,s]}$, (see fig. ~\ref{Fig.5} (f)-(h)). Note that this will 
result in a possible discontinuity at $s=\frac{1}{2}$. However, we will fill this discontinuity by the 
family of curves defined in the paragraph above, (see fig. ~\ref{Fig.5} (e)). Thus, the curves connecting $\sigma_0(\frac{1}{2})$
and $\sigma_1(\frac{1}{2})$ will not be unique and we can take $s=f(r)$, $r \in [0,1]$ for some 
non decreasing function $f(r)$. 
%Note that the length of the curves in the resulting family will be bounded by  $l_2+l_3+m_1+m_2$.

Figure ~\ref{Fig.5} demonstrates a family $\gamma_s$ continuously connecting the "opposite" points of $\sigma_1, \sigma_2$ starting with $e_0$ and ending with $e_1$. 

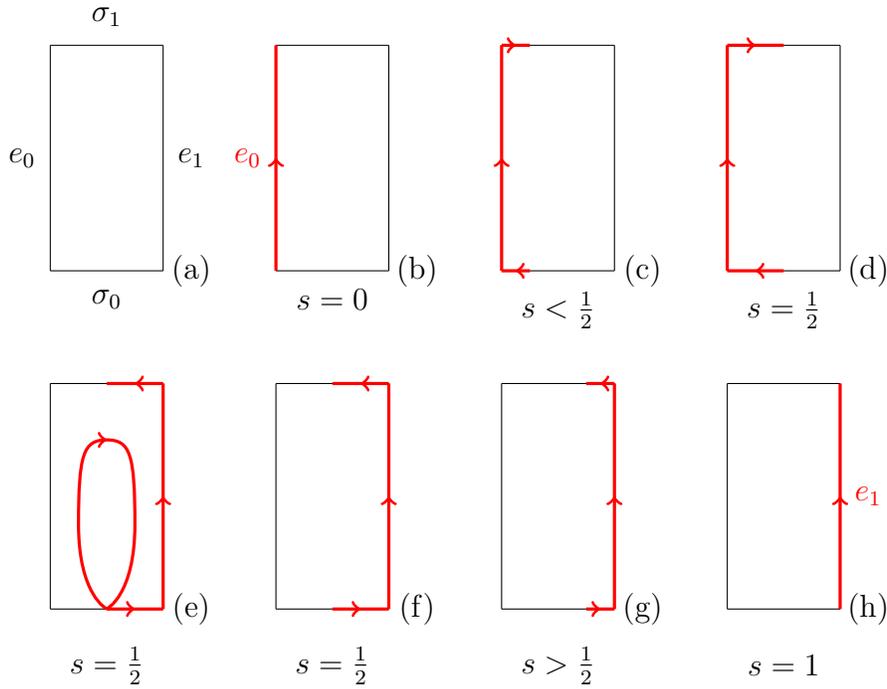
\begin{figure}[h!] 
    \centering
    \begin{tikzpicture} [scale=.75]
    
%1stline

\draw (-7,5) to (-5,5);
\draw (-3,5) to (-1,5);
\draw (-7,5) to (-7,1);
\draw (-5,5) to (-5,1);

\draw (-7,1) to (-5,1);
\draw (-3,1) to (-1,1);
\draw (-3,1) to (-3,5);
\draw (-1,1) to (-1,5);

\draw (7,1) to (5,1);
\draw (3,1) to (1,1);
\draw (3,1) to (3,5);
\draw (5,1) to (5,5);

\draw (7,5) to (5,5);
\draw (3,5) to (1,5); 
\draw (1,1) to (1,5);
\draw (7,1) to (7,5);

%2ndline

\draw (-7,-1) to (-5,-1);
\draw (-3,-1) to (-1,-1);
\draw (-7,-1) to (-7,-5);
\draw (-5,-1) to (-5,-5);

\draw (-7,-5) to (-5,-5);
\draw (-3,-5) to (-1,-5);
\draw (-3,-5) to (-3,-1);
\draw (-1,-5) to (-1,-1);

\draw (7,-5) to (5,-5);
\draw (3,-5) to (1,-5);
\draw (3,-5) to (3,-1);
\draw (5,-5) to (5,-1);

\draw (7,-1) to (5,-1);
\draw (3,-1) to (1,-1); 
\draw (1,-5) to (1,-1);
\draw (7,-5) to (7,-1);

%labels
\node at (-4.5,1) {(a)};
\node at (-0.5,1) {(b)};
\node at (3.5,1) {(c)};
\node at (7.5,1) {(d)};

\node at (-4.5,-5) {(e)};
\node at (-0.5,-5) {(f)};
\node at (3.5,-5) {(g)};
\node at (7.5,-5) {(h)};

\node at (-7.5,3) {$e_0$};
\node at (-4.5,3) {$e_1$};
\node at (-6,0.5) {$\sigma_0$};
\node at (-6,5.5) {$\sigma_1$};

\node[red] at (-3.5,3) {$e_0$};

\begin{scope}[very thick,decoration={
    markings,
    mark=at position 0.5 with {\arrow{>}}}
    ] 
    \draw[postaction={decorate},red] (-3,1) to (-3,5);
\end{scope}

\node at  (-2,0.5) {$s=0$};

\begin{scope}[very thick,decoration={
    markings,
    mark=at position 0.5 with {\arrow{>}}}
    ] 
    \draw[postaction={decorate},red] (1.5,1) to (1,1);
    \draw[postaction={decorate},red] (1,1) to (1,5);
    \draw[postaction={decorate},red] (1,5) to (1.5,5);
\end{scope}

\node at  (2,0.3) {$s<\frac{1}{2}$};

\begin{scope}[very thick,decoration={
    markings,
    mark=at position 0.5 with {\arrow{>}}}
    ] 
    \draw[postaction={decorate},red] (6,1) to (5,1);
    \draw[postaction={decorate},red] (5,1) to (5,5);
    \draw[postaction={decorate},red] (5,5) to (6,5);
\end{scope}
\node at  (6,0.3) {$s=\frac{1}{2}$};

\begin{scope}[very thick,decoration={
    markings,
    mark=at position 0.5 with {\arrow{>}}}
    ] 
    \draw[postaction={decorate},red] plot [smooth,tension=1.5] coordinates {(-6,-5) (-6.5,-3.5)  (-6,-2)  (-5.5,-3.5)  (-6,-5)};
    \draw[postaction={decorate},red] (-6,-5) to (-5,-5);
    \draw[postaction={decorate},red] (-5,-5) to (-5,-1);
    \draw[postaction={decorate},red] (-5,-1) to (-6,-1);
\end{scope}
\node at  (-6,-6) {$s=\frac{1}{2}$};

\begin{scope}[very thick,decoration={
    markings,
    mark=at position 0.5 with {\arrow{>}}}
    ]
    \draw[postaction={decorate},red] (-2,-5) to (-1,-5);
    \draw[postaction={decorate},red] (-1,-5) to (-1,-1);
    \draw[postaction={decorate},red] (-1,-1) to (-2,-1);
\end{scope}
\node at  (-2,-6) {$s=\frac{1}{2}$};

\begin{scope}[very thick,decoration={
    markings,
    mark=at position 0.5 with {\arrow{>}}}
    ]
    \draw[postaction={decorate},red] (2.5,-5) to (3,-5);
    \draw[postaction={decorate},red] (3,-5) to (3,-1);
    \draw[postaction={decorate},red] (3,-1) to (2.5,-1);
\end{scope}
\node at  (2,-6) {$s>\frac{1}{2}$};

\begin{scope}[very thick,decoration={
    markings,
    mark=at position 0.5 with {\arrow{>}}}
    ]
    \draw[postaction={decorate},red] (7,-5) to (7,-1);
\end{scope}
\node at  (6,-6) {$s=1$};
\node[red] at (7.5,-3) {$e_1$};

\end{tikzpicture}
\caption{Constructing $\gamma_s$}
\label{Fig.5}

\end{figure}

We observe that the same construction can be done parametrically.

%For $s \in [0, \frac{1}{2})$, $\sigma_0(s)$ will be connected to $\sigma_1(s)$ by $\sigma_0|_{[s,\frac{1}{2}]} * b_s *\bar{\sigma}_1 |_{[s, \frac{1}{2}]}$. For $s \in (\frac{1}{2},1]$ the corresponding points will be connected by $\bar{\sigma}_0|_{[\frac{1}{2},s]} * b_s * \sigma_1 |_{[\frac{1}{2}, s]}$. These are curves of length at most 
%$l_2+l_3+2 (m_1+m_2)$. Moreover, at $s=0$ we get the curve
%$\sigma_0|_{[0, \frac{1}{2}]} * \bar{\sigma}_0|_{[0, \frac{1}{2}]} * \sigma_1|_{[0, \frac{1}{2}]}*\bar{\sigma}_1|_{[0, \frac{1}{2}]}$, which is path-homotopic to $e_0$. 
%Likewise, the curve $b_1$ is path-homotopic to $e_1$ without increasing their lengths. 
%We will denote the resulting family of curves as $\gamma_s$.
\end{proof}

In this paper we will apply this lemma when $m_1, m_2$ are smaller than some small $\delta$. In this case we will conclude that lengths of the  $1$-parameter family of curves constructed above is at most
$l_3+l_2+R(\delta)$. 

\begin{definition} A one-parameter family of curves $G(s,t) = e_s(t)$, (see Fig. ~\ref{Fig.11}), $s, t \in I$ will be called a long  $\delta$-thin strip if the lengths of all transversal curves $\sigma_t(s)=G(s,t)$ are smaller than $\delta$ for all $t \in I$. We call the strip long, because  for each $s \in I$ we don't have a control over the lengths of curves $e_s(t)$. 
\end{definition}

\begin{figure}[h!] 
    \centering
\begin{tikzpicture}[thick]

    \draw [blue] (-5,1)--(-5,-1);
    \draw [blue] (5,1)--(5,-1);
    
    \draw [blue] (-5,1)--(5,1);
    \draw [blue] (-5,-1)--(5,-1);
    \draw [blue] (-4.5,1)--(-4.5,-1);
    \draw [blue] (-4,1)--(-4,-1);
    \draw [blue] (-3,1)--(-3,-1);
    \draw [blue] (-3.5,1)--(-3.5,-1);
    \draw [blue] (-2,1)--(-2,-1);
    \draw [blue] (-2.5,1)--(-2.5,-1);
    \draw [blue] (-1,1)--(-1,-1);
    \draw [blue] (-1.5,1)--(-1.5,-1);      
    \draw [blue] (-0.5,1)--(-.5,-1);
    \draw [blue] (0,1)--(0,-1);
    \draw [blue] (0.5,1)--(0.5,-1);
    \draw [blue] (4,1)--(4,-1);
    \draw [blue] (4.5,1)--(4.5,-1);
    \draw [blue] (3.5,1)--(3.5,-1);
    \draw [blue] (3,1)--(3,-1);
    \draw [blue] (2.5,1)--(2.5,-1);
    \draw [blue] (2,1)--(2,-1);
    \draw [blue] (1.5,1)--(1.5,-1);
    \draw [blue] (1,1)--(1,-1);

    \node at (1.5,-1.8) {$\sigma_t(s),\quad \textrm{length}(\sigma_t(s))<\delta$};
    \draw[->] (-1,-1.8) to [bend left = 50] (-1.5,-1.2);

    \end{tikzpicture}
    %\caption{$\delta$-thin strip}
\caption{$\delta$-thin strip $F:I \longrightarrow \Omega M$}
\label{Fig.11}

\end{figure}
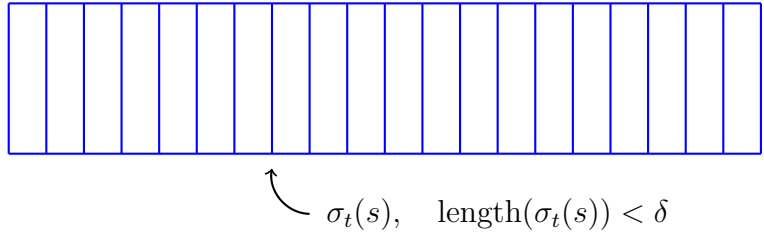

%\begin{figure}[h]
%    \centering
%     \input{Fig3.tikz}
%\end{figure}
\begin{definition}

%In general, let $\Omega M$ be the space of piecewise differentiable closed curves on $M$.
%Consider $F:I^m \longrightarrow \Omega M$. 
%Let $S_t=\{e_x(t), s \in I^m\}$ be the set of cross sections
%of $F$ in $M$ corresponding to fixed $t \in I$, parametrizing the curves $F(x)(t)$, $ x \in I^m$. 
  $F: I^m \longrightarrow \Pi M$ will be called a long $\delta$-thin strip, 
if there exists a continuous family of curves $\alpha_{\tilde{x}}$,  connecting $\tilde{x} \in \partial I^m$ with the point $\tilde{p}=\left(\frac{1}{2},...,\frac{1}{2} \right)$ with the property
that for each $t \in [0,1]$ the curve $S_t: \alpha_{\tilde{x}} \to M$
defined by $S_t(x) = F(x)(t)$ has length less than $\delta$
(see Figure \ref{Fig.1}). 
\end{definition}

\begin{figure}[h!] 
    \centering
\begin{tikzpicture}[thick, scale=0.50] 
	
	%squareouter
	
	\draw (-7,-2) to (-3,-2);
	\draw (-7,2) to (-3,2); 
	\draw (-7,2) to (-7,-2);
	\draw (-3,2) to (-3,-2);
	
	%boxouter
	
	\draw[very thick] (3,3.5) to (3,-3.5);
	\draw[very thick] (6,3.5) to (6,-3.5);
	\draw[very thick] (3,3.5) to (6,3.5);
	\draw[very thick] (3,-3.5) to (6,-3.5);
	
	\draw [very thick, densely dashed] (5,3.5) to (5,-1.5);
	\draw (5,5.5) to (5,3.5);
	
	\draw[very thick]   (8,5.5) to (8,-1.5);
	\draw[very thick]   (5,5.5) to (8,5.5);
	
	\draw[very thick, densely dashed] (5,-1.5) to (8,-1.5);
	\draw[very thick]  (6,-1.5) to (8,-1.5);
	
	\draw[very thick] (3,3.5) to (5,5.5);
	\draw [very thick, densely dashed] (3,-3.5) to (5,-1.5);
	
	\draw[very thick] (6,3.5) to (8,5.5);
	\draw[very thick] (6,-3.5) to (8,-1.5);

	%insidethebox
	
	\draw (3,.5) to (6,.5);
	\draw   (5,2.5) to (8,2.5);
	\draw  (3,.5) to (5,2.5);
	\draw  (6,.5) to (8,2.5);
	
	\filldraw (5.5,1.5) circle[radius=2.5pt];
	
	\draw  (5.5,1.5) to (3,.5);
	\draw  (5.5,1.5) to (4,.5);
	\draw  (5.5,1.5) to (5,.5);
	\draw  (5.5,1.5) to (6,.5);
	
	\draw  (5.5,1.5) to (3.7,1.2);
	\draw  (5.5,1.5) to (4.5,2);
	\draw  (5.5,1.5) to (5,2.5);

	\draw  (5.5,1.5) to (5.8,2.5);
	\draw  (5.5,1.5) to (7,2.5);
	\draw (5.5,1.5) to (8,2.5);
	
	\draw (5.5,1.5) to (6.7,1.2);
	\draw(5.5,1.5) to (7.3,1.8);

    % labels
    
             %square
    	\filldraw (-5,0) circle[radius=2.5pt];
    	\node at (-4.5, -.2)  { $\tilde p$};
    	\filldraw (-7,-1)  circle[radius=2.5pt];
    	\draw (-5,0) to  (-7,-1);
    	\node at (-6, .2)  { $ \alpha_{\tilde{x}}$};
        \node at (-7.5, -1)  { $ \tilde{x}$};
        \node at (-5, -3)  {\large $I^m$};
    	
              %arrow
    	\draw[->] (-2,0) to (2,0);
    	\node at (0, -.65)  {\large $F$};
    	
    	      %box
    	
        \node at (2.55,.7) { $S_t$};
        
\end{tikzpicture}
\caption{$\delta$-thin strip $F:I^2 \longrightarrow \Omega M$}
\label{Fig.1}
\end{figure}
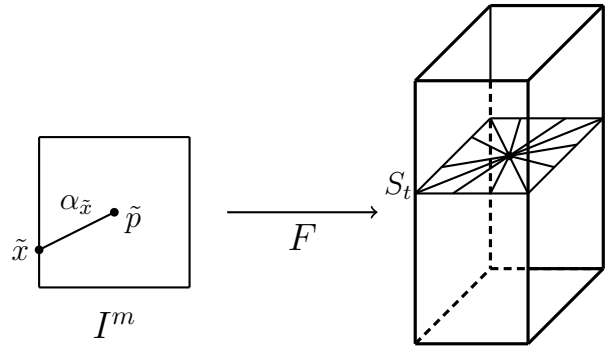

%\begin{figure}[h!] 
%    \centering
%\input{Fig12.tikz}
%\caption{Typical loops}
%\label{Fig.12}
%\end{figure}

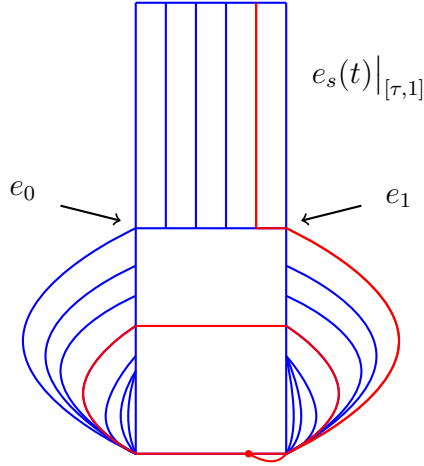
\begin{figure}[h!] 
    \centering
\begin{tikzpicture}[thick]

% Outer box
\draw [blue] (-1,3) -- (1,3);
\draw [blue] (-1,0) -- (1,0);
\draw [blue] (-1,-3) -- (1,-3);
\draw [blue] (-1,3) -- (-1,-3);
\draw [blue] (1,3) -- (1,-3);

% Red main arcs
\draw [red] plot [smooth, tension=1] coordinates {(1,-3) (2.5,-1.5) (1,0)};
\draw [blue] plot [smooth, tension=1] coordinates {(-1,-3) (-2.5,-1.5) (-1,0)};

% Right blue and red curves
\foreach \a/\b/\c in {
  2.2/-1.7/-0.5,
  2.0/-1.9/-0.9,
  1.7/-2.2/-1.3,
  1.4/-2.5/-1.7,
  1.2/-2.5/-1.7,
  1.1/-2.5/-1.7
} {
  \draw [blue] plot [smooth, tension=1] coordinates {(1,-3) (\a,\b) (1,\c)};
}

% Left blue and red curves
\foreach \a/\b/\c in {
  -2.2/-1.7/-0.5,
  -2.0/-1.9/-0.9,
  -1.7/-2.2/-1.3,
  -1.4/-2.5/-1.7,
  -1.2/-2.5/-1.9,
  -1.1/-2.5/-1.9
} {
  \draw [blue] plot [smooth, tension=1] coordinates {(-1,-3) (\a,\b) (-1,\c)};
}

% Middle lines
\draw [blue] (-0.6,3) -- (-0.6,0);
\draw [blue] (-0.2,3) -- (-0.2,0);
\draw [blue] (0.2,3) -- (0.2,0);
\draw [red] (0.6,3) -- (0.6,0);
\draw [red, thick] (0.6,0) -- (1,0);
\filldraw[red] (0.5,-3) circle (1pt);

% Decorations (arrows)
\begin{scope}[decoration={
    markings,
    mark=at position 0.5 with {\arrow{>}}}, postaction={decorate}
]
  \draw [red] (0.6,0) -- (0.6,3);
  \draw [red] plot [smooth, tension=1] coordinates {(1,-3) (2.5,-1.5) (1,0)};
  \draw [red] plot [smooth, tension=1] coordinates {(1,-1.3) (1.7,-2.2) (1,-3)};
  \draw [red] plot [smooth, tension=1] coordinates {(-1,-3) (-1.7,-2.2) (-1,-1.3)};
  \draw [red] (-1,-1.3) -- (1,-1.3);
  \draw [red] (1,-3) -- (-1,-3);
  \draw [red] plot [smooth, tension=1] coordinates {(0.5,-3) (0.8,-3.1) (1,-3)};
\end{scope}

% Labels
\node at (2.5,0.4) {$e_1$};
\draw[->] (2,0.3) -- (1.2,0.1);

\node at (-2.5,0.5) {$e_0$};
\draw[->] (-2,0.3) -- (-1.2,0.1);

\node at (2.1,2) {$e_s(t)\big|_{[\tau, 1]}$};

\end{tikzpicture}
\caption{Homotopy between $F$ and $\tilde{F}$}
\label{Fig.10}

\end{figure}

\begin{lemma} \label{lem:family of arcs}
Let $F: I \longrightarrow \Omega M, F(s)=e_s(t)=e_s$ be a long $\delta$-thin strip.  Let $a^0_t$, $a^1_t$ be two one-parameter families of curves with the following properties:
\begin{itemize}
  \item% $a^0_0= e_0(0)$ and $a^1_0 = e_1(0)$ are constant curves;
$a^i_0= e_i(0)$;
    \item % $a^0_r(0) = e_0(0)$ and $a^1_z(0) = e_1(0)$;
    $a^i_t(0) = e_i(0)$;
    \item %$a^0_r(1) = e_0(f(r))$ and $a^1_z(1) = e_1(g(z))$, where the functions $f(r), g(z)$ are some non-decreasing continuous functions of $r$ and $z$, with $f(0)=g(0)=0$ and $f(1)=g(1)=1$.
    %, (see fig. ~\ref{Fig.12}).
    $a^i_t(1) = e_i(t)$.
\end{itemize}
%connecting the points of $e_0$ with $e_0(0)$ and $e_1$ with $e_1(0)$ respectively, where the functions $f(r)=a^0_r(1), g(z)=a^1_z(1)$ are some non-decreasing continuous functions of $r$ and $z$, with $f(0)=g(0)=0$ and $f(1)=g(1)=1$, (see fig. ~\ref{Fig.12}). 
Suppose the maximal length of curves in both families is majorized by $L$.
    Then there exists a continuous family of curves $H: I \times I \longrightarrow  \Pi_{3L+R(\delta)}M$,
    % and a family of surjective non-decreasing maps $\{f_x: I \rightarrow I \}_{x \in I}$,
    such that 
    \begin{itemize}
        \item $H(x,i)= a^i_x$ for $i=0,1$.
        \item $H(x,t)$ is an arc of length $\leq 3L + R(\delta)$
        connecting $F(x)(0)$ to $F(x)(t)$.
    \end{itemize}
\end{lemma}

\begin{proof}
Fix $t \in (0,1)$. 
%Since $f$ and $g$ are non-decreasing surjective piecewise smooth functions there exist piecewise smooth functions $h_i:[0,1] \rightarrow [0,1]$, $i=1,2$, with $f(h_1(t)) = g(h_2(t))$ for all $t \in [0,1]$ and $h_i(0) = 0$, $h_i(1) =1$. Let $t'=f(h_1(\tilde{t})) = g(h_2(\tilde{t}))$.
Consider a quadrilateral formed by
curves $a^0_{t}$, $e_{[0,1]}(t)$, $a^1_{t}$ and  $e_{[0,1]}(0)$. 
%Note that $r'$ and $z'$ are chosen so that $a^0_{r'}(1) = e_0(t')$ and $a^1_{z'}(1) = e_1(t')$.
Applying Lemma \ref{shortcurves} we obtain a family
of arcs $H(t,x)$ interpolating between $a^0_{t}$
and $a^1_{t}$.
The result for all $t$ follows by applying the parametric version of Lemma \ref{shortcurves}.
\end{proof}

%\begin{figure}[h!]
%\centering	
%\input{Fig2.tikz}
%\end{figure}

The statement
of Lemma %\ref{shorteningstriplemma}
\ref{lem:family of arcs}
can be generalized to higher dimensions in the following lemma. 

\begin{lemma} \label{Generallemma}
Let $F:I^m \longrightarrow \Omega M$ be a long $\delta$-thin strip. Suppose 
%$\{f_{\tilde{x}}: [0,1] \to [0,1] \}_{\tilde{x} \in \partial I^m}$ is a family of surjective non-decreasing piecewise smooth functions and
$H:\partial I^m \times [0,1] \longrightarrow \Pi M$ is a family of paths of length at most $l$, such that each path begins at $x=F(\tilde{x})(0)$ and ends at 
$e_{\tilde x}(t)=F(\tilde{x})(t)$. Suppose also that for $\tilde{p} = (\frac{1}{2},...,\frac{1}{2})$ there exists a
one-parameter family of curves $a^{p}_t$ connecting point $p=F(\tilde{p})(0)$ with the points $e_p(t)=F(p)(t)$ of length at most $L$.  Then one can extend $H$ to $I^m \times [0,1] \longrightarrow \Pi M$ with the same properties, so that the length of paths is at most $l+2L+R(\delta)$. 
\end{lemma}

\begin{proof}

\begin{figure}[h!] 
    \centering
\begin{tikzpicture}[thick, scale=0.6]
	]
	%boxoutera

	\draw[very thick] (-3,3.5) to (-3,-3.5);
	\draw[very thick] (-6,3.5) to (-6,-3.5);
	\draw[very thick] (-3,3.5) to (-6,3.5);
	\draw[very thick] (-3,-3.5) to (-6,-3.5);

	\draw [very thick, densely dashed] (-4,3.5) to (-4,-1.5);
	\draw (-4,5.5) to (-4,3.5);
	
	\draw[very thick]   (-1,5.5) to (-1,-1.5);
	\draw[very thick]   (-4,5.5) to (-1,5.5);
	
	\draw[very thick, densely dashed] (-4,-1.5) to (-1,-1.5);
	\draw[very thick]  (-3,-1.5) to (-1,-1.5);
	
	\draw[very thick] (-6,3.5) to (-4,5.5);
	\draw [very thick, densely dashed] (-6,-3.5) to (-4,-1.5);
	
	\draw[very thick] (-3,3.5) to (-1,5.5);
	\draw[very thick] (-3,-3.5) to (-1,-1.5);
	
	%boxouterb
	
	\draw[very thick] (3,3.5) to (3,-3.5);
	\draw[very thick] (6,3.5) to (6,-3.5);
	\draw[very thick] (3,3.5) to (6,3.5);
	\draw[very thick] (3,-3.5) to (6,-3.5);
	
	\draw [very thick, densely dashed] (5,3.5) to (5,-1.5);
	\draw (5,5.5) to (5,3.5);
	
	\draw[very thick]   (8,5.5) to (8,-1.5);
	\draw[very thick]   (5,5.5) to (8,5.5);
	
	\draw[very thick, densely dashed] (5,-1.5) to (8,-1.5);
	\draw[very thick]  (6,-1.5) to (8,-1.5);
	
	\draw[very thick] (3,3.5) to (5,5.5);
	\draw [very thick, densely dashed] (3,-3.5) to (5,-1.5);
	
	\draw[very thick] (6,3.5) to (8,5.5);
	\draw[very thick] (6,-3.5) to (8,-1.5);
	
	%insidea
	
	\draw[very thick, pink] (-3.5,4.5) to (-3.5,-2.5);
	
	\draw[red] (-3.5,4.5) to [bend left = 70] (-3.5,-2.5);
	\draw[red] (-3.5,3.5) to [bend left = 60] (-3.5,-2.5);
	\draw[red] (-3.5,2.5) to [bend left = 50] (-3.5,-2.5);
	\draw[red] (-3.5,1.5) to [bend left = 50] (-3.5,-2.5);
	\draw[red] (-3.5,.5) to [bend left = 50] (-3.5,-2.5);
	\draw[red] (-3.5,.5) to [bend left = 25] (-3.5,-2.5);
	\draw[red] (-3.5,-1) to [bend left = 25] (-3.5,-2.5);
	
	\draw[very thick] (-5,3.5) to (-5,-3.5);
	\draw[red] (-5,3.5) to [bend right = 28] (-5,-3.5);
	\draw[red] (-5,2.25) to [bend right = 28] (-5,-3.5);
	\draw[red] (-5,1) to [bend right = 28] (-5,-3.5);
	\draw[red] (-5,-.25) to [bend right = 28] (-5,-3.5);
	\draw[red] (-5,-1.75) to [bend right = 28] (-5,-3.5);

    %insideb
    
    \draw[very thick] (5.5,4.5) to (5.5,-2.5);
    \draw[very thick] (4,3.5) to (4,-3.5);
    \draw[very thick] (4,3.5) to (5.5,4.5);
    
    \draw[red, very thick] (4,.5) to (5.5,1.5);
    \draw[red, very thick]  (4,-3.5) to (5.5,-2.5);
    \draw[red, very thick] (4,.5) to (4,-3.5);
    \draw[red, very thick] (5.5,1.5) to (5.5,-2.5);
    
    %labelsa
    
     \node at (-4.5, -4.5)  {\large $F(\tilde{x})(0)$};

     \node at (-4.5, -5.75) {\large (a)};
     \node at (4.75, -5.75) {\large (b)};
     
     \node at (.5, 3)  {\large $e_p(t)$};
     \draw[->] (-.25, 3) to[bend right=25] (-3.5,2.75);

     \node at (.3, 1.2)  {\large $a^p_{y}$};
     \draw[->]  (-.25, 1) to[bend right=25] (-1.4, 1);
     
     \node at (-8, 1)  {\large $e_{\tilde x}(t)$};
     \draw[->]  (-7.2, 1.1) to[bend left=25] (-5, 1.25);
     
      \node at (-8.25, -1)  {\large $H(\tilde{x})$};
       \draw[->]  (-7.5, -1) to[bend left=25] (-5.35, -1.25);

      \node[red] at (10, 1)  {\large $F(\alpha_{\tilde x})(t^\ast)$};
      \draw[->,red] (8.5, 1) to [bend left=25] (5,1);

      \node[red] at (10, -3)  {\large $F(\alpha_{\tilde x})(0)$};
      \draw[->,red] (8.5, -3) to [bend left=25] (5,-3);

	\end{tikzpicture}
\caption{Extending $H$}
\label{Fig.2}
\end{figure}
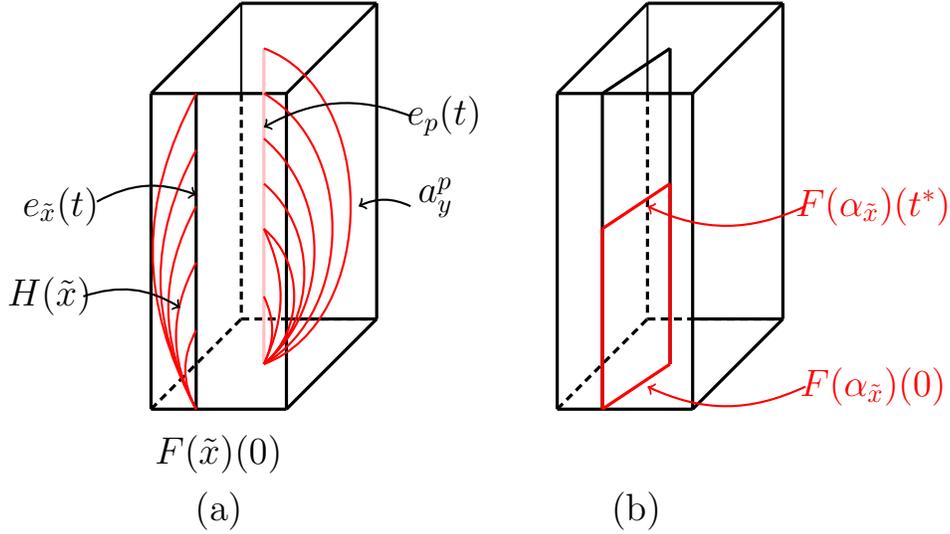

For each $\tilde{x} \in \partial I^m$, consider the long $\delta$-thin strip, with the boundary formed by the curves $e_x(t)|_{t \in [0,t^*]}$, $F(\alpha_{\tilde{x}})(0)$, $F(\alpha_{\tilde{x}})(t^*)$ and $e_p(t)|_{t \in [0,t^*]}$, (see fig. ~\ref{Fig.2} (a)). Next consider a quadrilateral formed by  $H(\tilde{x}, t^*)$,
$F(\alpha_{\tilde{x}})(t^*)$, 
$a^p_{t^*}$ and $F(\alpha_{\tilde{x}}) (0)$, (see fig. ~\ref{Fig.2} (b)). To each such quadrilateral apply 
Lemma \ref{lem:family of arcs}.
%~\ref{shorteningstriplemma}. 
This will give us the required family of curves. 
Note that the resulting continuity of the family of curves follows from continuity of 
$F(\alpha_{\tilde{x}})(t)$ with respect to both $x$ and $t$, which results in continuity of the quadrilaterals. The continuity of $H$ implies that the homotopies that are constructed in 
Lemma ~\ref{lem:family of arcs} also vary continuously. 
\end{proof}

Now we will use the Lemmas above
to prove Proposition \ref{mainprop: short arcs}.

Let us partition each rectangle of $R_i$ of $K$ into $N$ small rectangles $rec_k$ by first subdividing  each $e_i$ of $T$ into small sub-intervals using a partition $0=x_0 < x_1 < ... < x_s=1$ and next further subdividing each $\Sigma_{x_r}$ $r \in \{0,...,s\}$ so that  
each 
$F|_{rec_k}$ is a $\delta$-thin long strip for all $k \leq N$. 

The construction of $H$ will be by induction on the skeleta of $K$. 

Let $v_i$ be a vertex in the subdivision of $K$ that lies in the interior of $K$. $F(v_i)=\gamma_i(t)$, where $t \in [0,1]$ with $\gamma(0)=\gamma(1)=p_i$. 
Since each $\gamma_i$ lies on an $L$-short sphere there exists 
a family of arcs $a^i_t$ with $a^i_0 = a^i_1$ equal to constant curves
with image $F(v_i)(0)$, and $a^i_t$ connecting $F(v_i)(0)$ to 
$F(v_i)((t))$.
We define $H(v_i,t) = a^i_t$.

Applying Lemma \ref{lem:family of arcs} to each edge $[v_i, v_j]$  we 
extend $H$ to the 1-skeleton of the triangulation. Finally, applying
Lemma \ref{Generallemma} for each $rec_k$ we extend $H$
the 2-skeleton.
This finishes the proof of Proposition \ref{mainprop: short arcs}.

\bibliography{bib} 
\bibliographystyle{amsalpha}

\end{document}